\documentclass[12pt]{article}
\usepackage{graphicx} 
\usepackage{amsfonts}
\usepackage{amsmath}
\usepackage{amsthm}
\usepackage{amssymb}
\usepackage[hidelinks]{hyperref}
\usepackage{thmtools}
\usepackage{mathtools}
\usepackage{cases}
\usepackage{xcolor}
\usepackage{comment}
\usepackage{enumitem}
\usepackage[normalem]{ulem}
\usepackage[utf8]{inputenc}
\usepackage{tikz}
\usepackage{here}
\everymath{\displaystyle}
\usepackage{mathabx}
\numberwithin{equation}{section}

\usepackage{hyperref,cleveref} 
\usepackage[left=2cm,right=2cm,top=2.5cm,bottom=2.5cm]{geometry}

\usepackage{setspace}

\newcommand{\bU}{\mathbf{U}}
\newcommand{\tbU}{\widetilde{\mathbf{U}}}
\newcommand{\bu}{\mathbf{u}}
\newcommand{\tbu}{\widetilde{\mathbf{u}}}
\newcommand{\bv}{\mathbf{v}}
\newcommand{\tbv}{\widetilde{\mathbf{v}}}
\newcommand{\boeta}{\boldsymbol{\eta}}

\newcommand{\bzeta}{\boldsymbol{\zeta}}

 \newcommand{\be}{\begin{equation}}
\newcommand{\ee}{\end{equation}}
\newcommand{\eps}{\varepsilon}

\newcommand{\cH}{\mathcal{H}}

\newcommand{\cA}{\mathcal{A}}

\newcommand{\mC}{\mathbb{C}}
\newcommand{\mR}{\mathbb{R}}

\newtheorem{theorem}{Theorem}[section]
\newtheorem{lemma}[theorem]{Lemma}
\newtheorem{remark}{Remark}
\newtheorem{proposition}[theorem]{Proposition}
\newcommand{\n}[1]{\|#1\|_{\infty}}

\newcommand{\D}{\mathcal{D}(\mathcal{A})}

\title{One dimensional wave equation with in-domain localized damping and Wentzell boundary conditions}
\author{
A.~Dahmani\footnote{Aalen University, Research Center for Complex Systems, Aalen,
Germany (\texttt{abdelhakim.dahmani@hs-aalen.de}).}, 
Y.~Chitour\footnote{Universit\'e Paris-Saclay, CNRS, CentraleSup\'elec, Laboratoire des signaux et syst\`emes, 91190 Gif-sur-Yvette, France (\texttt{yacine.chitour@l2s.centralesupelec.fr}).},
H.-M.~Nguyen\footnote{Sorbonne Universit\'e, Universits\'e Paris Cit\'e, CNRS, INRIA, Laboratoire Jacques-Louis Lions, LJLL, F-75005 Paris, France (\texttt{hoai-minh.nguyen@sorbonne-universite.fr}).},
and C.~Roman\footnote{Aix-Marseille Universit\'e, Laboratoire informatique et syst\`emes, Marseille, France (\texttt{christophe.roman@lis-lab.fr}).}
}

\date{\today}

\begin{document}
\maketitle
\begin{abstract}
    This paper is devoted to the exponential stability for one-dimensional linear wave equations with in-domain localized damping and several types of Wentzell (or dynamic) boundary conditions. In a quite general boundary setting, we establish the exponential decay of solutions towards the corresponding steady states. The results are obtained either by the multiplier method or spectral analysis in an $L^2$-functional framework, and then with input-to-state technics in an $L^p$-functional framework for $p \in (2,\infty)$.

\end{abstract}

\tableofcontents
\section{Introduction}
In recent decades, Partial Differential Equations (PDEs) have been increasingly used in modeling because of their ability to represent complex phenomena. Although they often yield more precise results than other methods, they are challenging and delicate to handle due to their infinite-dimensional nature.

Among the many areas where PDEs are heavily utilized, one particularly interesting application is the damping and suppression of vibrations. This area has generated a substantial body of research over the past decades, see e.g. \cite{russell1969linear, dafermos_wave_1970,chen1982mathematical, huang1988mathematical}

Most vibration phenomena are modeled using the wave equation and significant research has been devoted to investigate its well-posedness, stability, controllability, and observability, among other aspects, see e.g. \cite{bardos1992sharp,lions1988controlabilite, chen1981note, haraux2018nonlinear,komornik1994exact,zuazua2024exact,MR2169126} and references therein. Recently, increasing attention has been directed towards the wave equation subject to Wentzell boundary conditions from researchers in both engineering and mathematics due to its usefulness in real-world problems and the technical mathematical challenges it presents, see, e.g., (\cite{mercier2018indirect, fourrier2013regularity, mugnolo2011damped, buffe2017stabilization, cavalcanti2012geo, cavalcanti2012wave,chitour2024exponential}) and the references therein. 

The one-dimensional $(1D)$ case, in particular, has been extensively studied since it has been used in modelling of many applications.  For instance, the 1D wave equation with Wentzell boundary conditions has been utilized in \cite{conrad1998strong,d2000exponential} to analyze the stability of an overhead crane, in \cite{bohm2014modeling} to model a hanging cable immersed in water, and in \cite{vanspranghe2021velocity} to investigate drilling torsional vibrations. Further examples can be found in \cite{lee1987stabilization, morgul1994stabilization, terrand2019regulation, roman2018boundary, roman2022pi} and the references therein.

In \cite{chitour2023lyapunov}, one deals with the $(1D)$ wave equation on $[0,1]$ with Wentzell boundary conditions at both ends, considering constant disturbances both within the domain and at the boundary. Assuming an everywhere active damping, they achieved regulation with an exponential convergence rate using proportional-integral control. Their analysis relied on identifying an appropriate Lyapunov functional besides some further analysis. In addition, they showed the exponential stability of some related systems among which systems with no integral action on the dynamics boundary condition.

It is important to mention that boundary damping alone is not sufficient to achieve exponential stability, for the case of a Wentzell boundary condition at one end and Dirichlet at the other end, see \cite{lee1987stabilization} (the authors are unable to find this reference, however, its results were announced in \cite{morgul1994stabilization}, in addition, an available proof can be found in  \cite[Sec.4.]{li2017boundary}), for the case of Wentzell boundary conditions at both ends,  see \cite{conrad1998strong}. Therefore, internal damping is required to obtain exponential stability. This has been addressed in \cite{chitour2023lyapunov} where the damping is used in the entire domain, as explained in the previous paragraph. However, this approach may not be optimal, as suggested by the sufficiency and necessity of the geometric control condition (GCC) established in \cite{bardos1992sharp}. Notice that the GCC in this setting is equivalent to the fact that the damping is active on some non empty subinterval of the domain of the space variable. Aiming for an exponential stability result and motivated by the fact that engineers strive to reduce the human intervention in models, we would like to resume the work conducted in \cite{chitour2023lyapunov} while considering a localized damping (Assumption ($A2$) in the sequal to be compared to \cite[Assumption ($h_2$)]{chitour2023lyapunov}) and we will also consider the $L^p$ framework for $p\in (2,\infty)$.

More precisely, we consider the following problem
\begin{subnumcases}{\label{sys_main_int}}
u_{tt}(t, x) - (a (x)u_x(t, x))_x  = -q(x) u_t(t, x) \quad \mbox{ for } (t,x) \in \mathbb{R}_+ \times (0, 1), \label{sys_main_int:1} \\
u_t(t,1)=\eta_1(t) \quad \mbox{ for } t\geq 0, \label{eq:eta_1} \\
u_t(t,0)=\zeta_1(t) \quad\mbox{ for }  t\geq 0, \\
\dot\eta_1(t)=-\alpha_1\eta_1(t)-\alpha_2\eta_2(t)-\beta_1u_x(t,1) \quad \mbox{ for } t\geq 0, \label{eq:xi1}\\
\dot\eta_2(t)=\eta_1(t) \quad \mbox{ for } t\geq 0, \label{eq:eta}\\
\dot\zeta_1(t)=-\gamma_1\zeta_1(t)+\mu_1u_x(t,0) \quad \mbox{ for } t\geq 0,\label{eq:xi0}\\
u(0, \cdot) = u_0, \quad u_t (0, \cdot) = u_1 \quad \mbox{ in } (0, 1), \\
\eta_1(0)=\eta_{1,0},\quad \eta_2(0)=\eta_{2,0}, \quad \zeta_1(0)=\zeta_{1,0},  
\end{subnumcases}
where $\big(u_0(\cdot), u_1(\cdot), \eta_{1,0}, \eta_{2, 0}, \zeta_{1,0} \big) $ is the initial data belonging to an appropriate space given later.
Here are the assumptions on the coefficients $a, q, \alpha_1, \alpha_2, \beta_1, \gamma_1$, and  $\mu_1$:

\begin{enumerate}
\item $a$ is a Lipschitz positive function defined in $(0, 1)$ bounded below and above by a positive constant, i.e.,  
\[
(A1)\quad a \in W^{1,\infty}(0,1) \textrm{ and there exist }
\underline{a},\ \overline{a}>0 \textrm{ s.t. }\underline{a}\leq a(\cdot) \leq \overline{a}
\textrm{ a.e. in }[0,1]. 
\]
\item $q$ is a bounded non-negative function defined in $[0,1]$  and is bounded below by a positive constant on some open interval $\omega$ of $(0, 1)$, i.e., 
\[
(A2)\quad q \in L^\infty(0, 1),\ \  q \ge 0\textrm{ a.e. in }(0, 1),\ 
    q(x)\in [\underline{q},\overline{q}],\quad  \mbox{a.e. in } \omega \subset  [0,1],
\]
for some  $\underline{q}, \overline{q}>0$
\item $\alpha_1$, $\alpha_2$, $\beta_1$, $\gamma_1,\mu_1$ are constants such that  
\[
(A3)\quad\alpha_1, \alpha_2, \beta_1, \gamma_1, \mu_1 > 0.
\]
\end{enumerate}

We have, by \eqref{eq:eta_1} and \eqref{eq:eta}, 
\be
\frac{d}{dt} \big(u(t, 1) - \eta_2(t) \big) = \eta_1(t) - \eta_2'(t) = 0. 
\ee
It follows that $u(t, 1) - \eta_2(t)$ is constant along every trajectory. Denote 
\begin{equation}\label{u_*}
    u_*=u_0(1)-\eta_{2,0}.  
\end{equation}
Then 
\be \label{prop-u_*}
u(t,1)-\eta_2(t) = u_* \mbox{ for } t \ge 0. 
\ee

The paper aims to establish exponential decay of the energy of the solutions and prove that the solutions converge to the system's attractor, which is $(u_*, 0, 0, 0)$. 

This paper is organized as follows: In Section 2, we define the energy functional associated with the problem in the $L^p$-functional framework for $1 \le p < + \infty$ and state the main results of the paper. Section 3 is devoted to establishing the exponential decay in the $L^2$-setting. Two approaches are proposed: one is based on the multiplier technique, and one is based on the spectral theory of a strongly continuous semigroup.
In Section 4, we discuss other systems with Wentzell boundary conditions as considered in \cite{chitour2023lyapunov} and we demonstrate their exponential stability. Finally, in Section 5, we establish the exponential decay in the $L^p$-setting for $p > 2$.

\section{Statement of the main results}
We first introduce the energy of the solution of the system, for $1 \le p < + \infty$, 
\begin{align}\label{def_E}
    E_p(t)=E_{p, i} (t)+E_{p, b}(t),
\end{align}
where $E_{p, i}$ and $E_{p, b}$ are respectively the energy related to the interior and to the boundary of the domain and are defined by  
\begin{align}
    E_{p, i}(t)=\frac{1}{p}\int_0^1 (|u_t(t, x)|^p+a(x) |u_x(t, x)|^p) \, dx,
\end{align}
\begin{align}
    E_{p, b}(t)= \frac{a(1)}{p\beta_1} |\eta_1(t)|^p+\frac{a(1) \alpha_2}{p\beta_1} |\eta_2(t)|^p+\frac{a(0)}{p\mu_1} |\zeta_1(t)|^p.  
\end{align}
When $p=2$, we simply denote $E_p$, $E_{p, i}$, and $E_{p, b}$ by $E$, $E_{i}$, and $E_{b}$ for notational ease.

Simple computations show that, for $p=2$, 
\begin{align}\label{ELocdot}
\dot{E}(t) = -\int_0^1 qu_t^2 \, dx-\frac{a(1)\alpha_1}{\beta_1}\eta_1^2-
\frac{a(0)\gamma_1}{\mu_1}\zeta_1^2.
\end{align}
Thus $E(t) = E_2(t)$ is a decreasing function for $t \ge 0$.

It is known that for $p=2$ and for $(u_0, u_1, \eta_{1, 0}, \eta_{2, 0}, \xi_{1, 0}) \in H^1(0, 1) \times L^2 (0, 1) \times \mR^3$, there exists a unique weak solution (in the sense given by a strongly continuous semigroup)
\begin{equation*}
    (u, \eta_1, \eta_2, \zeta_1) \in X \times \big((C[0, + \infty);\mathbb{R}) \big)^3
\end{equation*}
where 
$$
X = X_2 = C^0([0, +\infty); H^1(0,1)) \cap C^1([0,+\infty); L^2(0,1))
$$
(see, e.g., \Cref{wellposedness} in \Cref{sec:frequency}). 

\medskip 
Here is the first main result of the paper concerning the case $p=2$.

\begin{theorem}\label{th:mainth} Let $p=2$ and assume that $(A1)$, $(A2)$, and $(A3)$ hold true. Then there exist two positive constants $M$ and  $\nu$ such that
for every initial condition $(u_0, u_1, \eta_{1, 0}, \eta_{2, 0}, \xi_{1, 0}) \in H^1(0, 1) \times L^2(0, 1) \times \mR^3$, it holds for the corresponding unique weak solution of \eqref{sys_main_int} that, for $t \ge 0$, 
\begin{align}\label{th:mainth-cl1}
E(t)\leq M E(0)e^{- \nu t}
\end{align}
and 
\be \label{th:mainth-cl2}
\max_{x \in [0,1]} |u(t,x)-u_*|^2  \le M E(0)e^{- \nu t},
\ee
where $u_*$ is defined in \eqref{u_*}.
\end{theorem}

We present two approaches to prove \Cref{th:mainth} (see \Cref{sect-L2}). The first one relies on the use of multipliers, which, to the best of our knowledge, have not been previously utilized in the present context (Wentzell boundary conditions and localized damping). The second approach is based on the frequency domain analysis and more precisely on checking the assumptions of a theorem proved by Huang \cite{huang1985} and Prüss \cite{pruss1984}. The multiplier technique is also involved in this approach.

We next deal with the case $p>2$. In this case, we assume in addition that $a$ is constant on $(0, 1)$ and for notational ease this constant will be assumed to be $1$.
Given an initial condition $(u_0, u_1, \eta_{1, 0}, \eta_{2, 0}, \xi_{1, 0}) \in W^{2, p}(0, 1) \times W^{1, p}(0, 1) \times \mR^3$, one can show that (see \Cref{prop:existence-p} and \Cref{pro-WP}) that there exists a unique weak solution 
\begin{equation*}
    (u, \eta_1, \eta_2, \zeta_1) \in X_{p} \times \big(C^0([0, \infty];\mathbb{R}) \big)^3
\end{equation*}
of \eqref{sys_main_int}, where 
$$
X_{p} = C^0([0, + \infty); W^{1, p}(0,1)) \cap C^1([0, +\infty); L^p(0,1)),
$$
such that 
\be
\partial_t u, \partial_x u \in C([0, 1]; L^p(0, T)) \mbox{ for all } T>0. 
\ee

Here is the main result in the case $p>2$.

\begin{theorem} \label{thm-Lp} Let $2 < p < + \infty$ and assume that (A1), (A2), and (A3) hold true, and $a \equiv 1$ in $[0, 1]$. Then there exist two positive constants $M$ and $\nu$ such that, for every initial condition $(u_0, u_1, \eta_0, \eta_{2, 0}, \xi_0) \in W^{1, p}(0,1) \times L^p(0, 1) \times \mR^3$, it holds for the corresponding unique weak solution
of \eqref{sys_main_int} that, for $t \ge 0$, 
\be\label{thm-Lp-cl1}
E_p(t) \le M e^{- \nu t} E_p(0)
\ee
and 
\be\label{thm-Lp-cl2}
\max_{x \in [0,1]} |u(t,x)-u_*|^p  \le M E_p(0)e^{- \nu t},
\ee
where $u_*$ is defined in \eqref{u_*}.
\end{theorem}

The starting point of the proof of \Cref{thm-Lp} is based on the existence of the Riemann invariant for $u$, a component of the system, in the case $a \equiv 1$. 
As far as we know, it is not known that the energy in $L^p$-scale is decreasing for $p \neq 2$. To overcome this issue, our analysis is based on three ingredients. The first one is the exponential decay in the $L^p$-functional framework of the following system 
\be \label{sys-W-Lp}
\left\{\begin{array}{c}
u_{tt}(t, x) - u_{xx} (t, x)  = - q(x) u_t(t, x) \quad \mbox{ for } (t,x) \in \mathbb{R}_+ \times (0, 1),  \\[6pt]
u_t(t,1)=0 \quad  \mbox{ for } t \in \mR_+, \\[6pt]
u_t(t,0)=0 \quad  \mbox{ for } t \in \mR_+. 
\end{array} \right.
\ee 
The analysis of this problem follows the same lines as in the one given in \cite{chitour2024exponential}. The second ingredient is the estimate of the solution of \eqref{sys-W-Lp} in the $L^p$-scale with inhomogeneous boundary conditions, i.e., $u_t(t, 1)$ and $u_t(t, 0)$ are not supposed to be 0. The results obtained from these two ingredients are given in \Cref{lem-R}. The third ingredient is the use of these two estimates and the one in the $L^2$-functional framework mainly on $(\eta_1, \eta_2, \zeta_1)$ to derive the exponential decay. It is the place where we require the condition $p>2$.

\section{Analysis in the $L^2$-setting - Proof of \Cref{th:mainth}} \label{sect-L2}

This section is devoted to the proof of \Cref{th:mainth}. Two proofs using different approaches are given. The first proof which is based on the multipliers technique is given in the first subsection. The second proof is based on the spectral theory of the infinitesimal generator of a strongly continuous semigroup is given in the second subsection.

\subsection{Multipliers technique approach}
In this subsection, we establish certain inequalities which, together with \cite[Theorem 8.1, p. 103]{komornik1994exact}, allow us to conclude the exponential decay of the energy towards zero. These inequalities are derived using appropriate multipliers and the fact that the problem possesses strong (differentiable) solutions.

The second assertion of Theorem \ref{th:mainth} follows from the first one. Indeed, we have for every $t\geq 0$ that
\begin{equation}
    \begin{aligned}
\left|u(t, x)-u_*\right|^2 & \leq 2|u(t, x)-u(t, 1)|^2+2 \eta_2^2(t)  \\
& \leq 2 \int_0^1 u_x^2(t, x)\, dx+2 \eta_2^2(t) \\
& \leq \frac{4}{\underline{a}} E(t)+2 \eta_2^2(t)\\
& \leq C E(t).
\end{aligned}
\end{equation}

Following the standard strategy, the proof of the first assertion in Theorem \ref{th:mainth} is performed first for strong solutions and then extended by density to weak solutions. 

Before the actual proof of the main result, we need several definitions and lemmas, which are stated next. Let us consider $\zeta_2$ defined as 
\begin{equation}
     \zeta_2(t):=u(t,0)-u_*.
\end{equation}

According to the choice of $\omega=[1-\varepsilon_3,1]$, we choose $x_0=0$ as an observation point. We consider for $0<\varepsilon_0<\varepsilon_1<\varepsilon_2<\varepsilon_3$, the sets $Q_i=\left(1-\varepsilon_i, 1+\varepsilon_i\right)$, $i=0,1,2$, as well as three smooth non negative functions $\psi, \phi$ and $\varphi$ (and taking values in $[0,1]$ for $\phi$ and $\varphi$) defined below and depicted according to Figure~\ref{fig_smooth}.

\begin{equation*}
\left\{\begin{array} { l } 
{ \psi = 0 \text { on } Q _ { 0 } , } \\
{ \psi = \frac{e^{Lx}-1}{\underline{a}Lx} \text { on } ( 0 , 1 ) \backslash Q _ { 1 } , }
\end{array} \quad \left\{\begin{array} { l } 
{ \phi = 1 \text { on } Q _ { 1 } , } \\
{ \phi = 0 \text { on } ( 0 , 1 ) \backslash Q _ { 2 } , }
\end{array} \quad \left\{\begin{array}{l}
\varphi=1 \text { on } Q_2 \cap(0,1), \\
\varphi=0 \text { on } \mathbb{R} \backslash \omega,
\end{array}\right.\right.\right.
\end{equation*}
where $L=\sup_{x\in[0,1]}\frac{\vert a'(x)\vert}{a(x)}$. The definition of $\psi$ yields after a simple computation that
\begin{equation}\label{eq:psi-L}
(ax\psi)_x\geq 1,\quad (x\psi)_x a\geq \frac{a}{\underline{a}} \geq 1,\quad \textrm{ on } 
[0 , 1]\backslash Q _ { 1 }.
\end{equation}

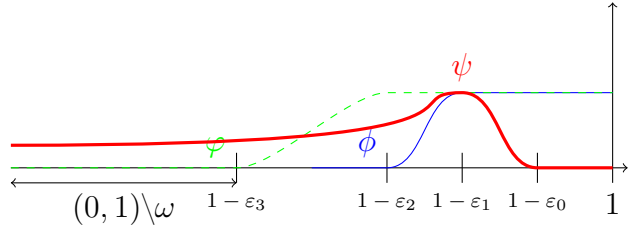
\begin{figure}[H]
\centering
\begin{tikzpicture}
\draw[->] (0,-.2)--(0,2.2) node[pos=0,below]{1};

\draw (-1,.2)--++(0,-.4) node[below]{\scriptsize$1-\varepsilon_0$};
\draw (-2,.2)--++(0,-.4) node[below]{\scriptsize$1-\varepsilon_1$};
\draw (-3,.2)--++(0,-.4) node[below]{\scriptsize$1-\varepsilon_2$};
\draw (-5,.2)--++(0,-.4) node[below]{\scriptsize$1-\varepsilon_3$};
\draw[->] (-8,0)--(.2,0) ;
\draw[blue] (-4,0)--(-3,0) node[above left]{$\phi$} .. controls (-2.5,0) and (-2.5,1) .. (-2,1) -- (0,1);
\draw[dashed,green] (-8,0)--(-5,0) node[above left]{$\varphi$} .. controls (-4.5,0) and (-3.5,1) .. (-3,1) -- (0,1);
\draw[<->] (-8,-.2)--(-5,-.2) node[pos=.5,below]{$(0,1)\setminus\omega$};
\draw[red,line width=1.2pt] (-8,.3).. controls (-1,.3 ) and (-3,1 ).. (-2,1) node[above]{$\psi$}.. controls (-1.5,1 ) and (-1.5,0 ) .. (-1,0)--(0,0);
\end{tikzpicture}
\caption{Illustration of the functions $\phi$, $\psi$, and $\varphi$.}
\label{fig_smooth}
\end{figure}

In all what follows $C$ denotes a generic positive constant that depends only on the data of the problem and also on $\psi, \phi$ and their derivatives.

The proof of the first assertion in Theorem \ref{th:mainth} is a consequence of a series of lemmas which are given below. 
\begin{lemma}\label{lem:S0}
Assume $(A1)$, $(A2)$, and $(A3)$ hold true. Then, there exists a positive constant $C$
 such that strong solutions of system \eqref{sys_main_int} satisfy, for every $0\leq S \leq T$, 
\begin{equation} \label{eslemma22}
\int_S^T  \Big( E_i(t)+ \eta_1^2(t) +\zeta_1^2(t)\Big) \, dt \leq C \underbrace{\int_S^T \int_{Q_1 \cap(0,1)}(u_t^2+au_x^2) \, dx  \, dt}_{\mathbf{S}_0} +CE(S).
\end{equation}
\end{lemma}
\begin{proof}
Multiplying the first equation of (\ref{sys_main_int}) by $x \psi a u_x$ then integrating over $[S, T] \times[0,1]$ yields, 
\begin{equation}\label{eeq26}
\int_S^T \int_0^1 x \psi au_x\left(u_{t t}-(au_{x})_x+q(x) u_t\right) \mathrm{d} x \mathrm{~d} t=0 .
\end{equation}
Starting with $\int_S^T \int_0^1 x \psi a u_x u_{tt} \, dx  \, dt$. One has
\begin{equation}\label{eeq27}
    \int_S^T \int_0^1 x \psi a  u_x u_{tt} \, dx  \, dt= -\int_S^T \int_0^1 x \psi a  u_{tx} u_{t} \, dx  \, dt +\left[\int_0^1x \psi a u_xu_t \, dx \right]_S^T.
\end{equation}
We integrate by parts the first term on the right-hand side with respect to $x$ to get
 \begin{equation}\label{eeq28}
      -\int_S^T \int_0^1 x \psi a u_{tx} u_{t} \, dx  \, dt = \frac{1}{2} \int_S^T \int_0^1 (x \psi a)_x u_t^2 \, dx  \, dt.
 \end{equation}
After integrating by parts with respect to $x$, we have
\begin{equation}\label{eeq29_1}
    -\int_S^T \int_0^1 x \psi a u_x (au_{x})_x \, dx  \, dt= \int_S^T \int_0^1 (x \psi)_x (a u_x)^2 \, dx  \, dt + \int_S^T \int_0^1 x \psi a u_x  (a u_x)_x \, dx  \, dt,
\end{equation}
from which we deduce that 
\begin{equation}\label{eeq29}
    -\int_S^T \int_0^1 x \psi a u_x (au_{x})_x \, dx  \, dt=\frac{1}{2} \int_S^T \int_0^1 (x \psi)_x a^2 u_x^2 \, dx  \, dt.
\end{equation}

Combining (\ref{eeq26}), (\ref{eeq27}), (\ref{eeq28}) and (\ref{eeq29}) leads to

\begin{equation}\label{eq:good-psi}
  \int_S^T \int_0^1 \big[ (x \psi a )_x u_t^2  + (x \psi)_x a^2 u_x^2 \big]\, dx  \, dt   =
  -2\left[\int_0^1 x \psi a u_xu_t \, dx \right]_S^T -2 \int_S^T \int_0^1 x \psi a u_x q(x)u_t \, dx  \, dt.
\end{equation}
We cut the first integral according to the partition 
$[0,1]=([0,1]\setminus Q_1)\cup Q_1$. Using \eqref{eq:psi-L}, we have 
\[
 (x \psi a )_x u_t^2  + (x \psi)_x a^2 u_x^2\geq u_t^2  +a^2 u_x^2,
 \quad \textrm{ for } x\in [0,1]\setminus Q_1,
 \]
 and hence $ \int_S^T \int_{[0,1]\setminus Q_1} \big[ (x \psi a )_x u_t^2  + (x \psi)_x a^2 u_x^2 \big]\, dx  \, dt\geq \int_S^T \int_{[0,1]\setminus Q_1} (u_t^2+a^2 u_x^2)\, dx  \, dt$. 

Using the above and adding to \eqref{eq:good-psi} the integral $\int_S^T \int_{Q_1} (u_t^2+a^2 u_x^2)\, dx  \, dt$, we get that there exists a positive constant such that
\begin{equation}\label{eeq32}
\begin{aligned}
   \int_S^T E_i(t) \, dt \leq &\ C \underbrace{ \int_S^T \int_{ Q_1 \cap (0,1) } (u_t^2+au_x^2)\, dx  \, dt}_{\textbf{S}_0} +C\underbrace{\left|\left[\int_0^1 u_xu_t \, dx \right]_S^T\right|}_{\textbf{T}_1}\\ & + C\underbrace{\int_S^T \int_0^1 | aqu_xu_t| \, dx  \, dt}_{\textbf{T}_2}.
\end{aligned}
    \end{equation}
We estimate now some terms on the right-hand side of \eqref{eeq32}. 
For $\textbf{T}_1$, by using Cauchy-Schwarz inequality, we obtain that
\begin{equation*}
   \left|\int_0^1 u_xu_t \, dx \right|\leq C\int_0^1 (u_t^2+au_x^2 ) \, dx\leq CE_i(t) \leq CE(t).
\end{equation*}
Then, since $E$ is a non-increasing function of $t$ we get
\begin{equation}\label{eeq34}
    \textbf{T}_1 \leq  CE(S)+ CE(T) \leq 2C E(S).
\end{equation}
Using  Young's inequality with $\delta_0>0$ and \eqref{ELocdot}, we find that
\begin{equation}\label{eeq35}
\begin{aligned}
        \textbf{T}_2 &\leq  \frac{\delta_0}{2}  \int_S^T \int_0^1 au_x^2 \, dx  \, dt + \frac{C}{2\delta_0} \int_S^T \int_0^1 q(x) u_t^2 \, dx  \, dt \\
       & \leq \delta_0  \int_S^T E_i(t)\, dt + \frac{C }{2\delta_0}E(S).
    \end{aligned}
\end{equation}
Moreover, from \eqref{ELocdot} we have
\begin{equation}\label{eq:eta1xi1}
    \int_S^T \Big(\eta_1^2(t) +\zeta_1^2(t)\Big) \, dt \leq C \int_S^T \big( -\dot{E}_2(t)\big) \, dt \leq CE(S).
\end{equation}
Gathering (\ref{eeq32}), (\ref{eeq34}), (\ref{eeq35}) and \eqref{eq:eta1xi1} besides taking $\delta_0$ sufficiently small, we conclude (\ref{eslemma22}).

 \end{proof}

\begin{lemma}\label{lemmaS0}
 Assume $(A1)$, $(A2)$, and $(A3)$ hold true. Then, there exists a positive constant $C>0$ such that strong solutions of system \eqref{sys_main_int} satisfy
    \begin{equation}\label{eq36}
\mathbf{S}_0 + \int_S^T \eta_2^2(t) \, dt \leq C\underbrace{\int_S^T \int_{Q_2 \cap(0,1)}|u-u_*|^2 \, dx  \, dt}_{\mathbf{S}_1}+C E(S),
\end{equation}
for every $0\leq S \leq T.$
\end{lemma}
\begin{proof}
    We multiply the first equation of $(\ref{sys_main_int})$ by $\phi  (u-u_*)$ and we integrate over $[S, T] \times[0,1]$ to obtain that
\begin{equation}\label{lemmaS0:eq37}
\int_S^T \int_0^1 \phi (u-u_*) \left(u_{tt}-(au_{x})_x+q(x)u_t\right) \, dx  \, dt=0.
\end{equation}
For the first term, we have
\begin{equation}\label{lemmaS0:eq38}
\int_S^T \int_0^1 \phi (u-u_*) u_{tt} \, dx  \, dt=- \int_S^T \int_0^1  \phi u_t^2  \, dx \, dt+ \left[\int_0^1 \phi (u-u_*) u_t \, dx\right]_S^T.
\end{equation}
For the second term, an integration by part with respect to $x$ yields
\begin{equation}\label{lemmaS0:eq39}
\begin{aligned}
-\int_S^T \int_0^1 \phi (u-u_*) (au_{x})_x \, dx  \, dt=&\int_S^T \int_0^1 (\phi (u-u_*))_x (a u_x) \, dx  \, dt  - \int_S^T  \big[\phi(u-u_*) (a u_x) \big]_0^1dt\\
=& \int_S^T \int_0^1\phi a u_x^2 \, dx  \, dt +\int_S^T \int_0^1\phi_x (u-u_*) a u_x \, dx  \, dt\\ 
&- \int_S^T a(1)\eta_2(t)u_x(t,1)\, dt.
\end{aligned}
\end{equation}
We now simplify the two last terms on the right-hand side of \eqref{lemmaS0:eq39}.
Using a suitable integration by parts we find that
\begin{equation} \label{neq33}
    \begin{aligned}
        \int_S^T \int_0^1  \phi_x (u-u_*) au_x \, dx  \, dt &= \frac{1}{2}\int_S^T \int_0^1 a \phi_x \big((u-u_*)^2\big)_x \, dx  \, dt \\
        &= - \frac{1}{2}\int_S^T \int_0^1 (a\phi_{x})_x (u-u_*)^2 \, dx  \, dt + \left[ \frac{1}{2}\int_S^T a \phi_x (u-u_*)^2  \, dt  \right]_0^1 \\
        &= -\frac{1}{2}\int_S^T \int_0^1 (a \phi_{x})_x (u-u_*)^2 \, dx  \, dt,
    \end{aligned}
\end{equation}
where the last term on the second line in \eqref{neq33} vanishes because $\phi$ is constant in some neighborhoods of $0$ and $1$.\\
As for the boundary term $-\int_S^T a(1) \eta_2(t)u_x(t,1)\, dt$, we multiply \eqref{eq:xi1} by $\frac{1}{\beta_1} \eta_2(t)$ and integrate over $[S,T]$ to obtain that
\begin{equation}\label{eeq37}
    \begin{aligned}
     -\int_S^T a(1)\eta_2 u_x(t, 1) \, dt=& \frac{a(1)\alpha_2}{\beta_1} \int_S^T\eta_2^2(t) \, dt -\frac{a(1)}{\beta_1} \int_S^T \eta_1^2(t)\, dt\\
     &+ \int_S^T\frac{d}{\, dt}\left(\frac{a(1) \alpha_1}{2 \beta_1} \eta_2^2(t)+ \frac{a(1)}{\beta_1}\eta_1(t) \eta_2(t)\right)\, dt\\    
    =& \frac{a(1) \alpha_2}{\beta_1} \int_S^T\eta_2^2(t) \, dt -\frac{a(1)}{\beta_1} \int_S^T \eta_1^2(t)\, dt\\
    &+ \left[\frac{a(1) \alpha_1}{2 \beta_1} \eta_2^2(t)+ \frac{a(1)}{\beta_1}\eta_1(t) \eta_2(t)\right]_S^T.
    \end{aligned}
\end{equation}
Moreover, for the third term in \eqref{lemmaS0:eq37} an integration by parts with respect to $t$ gives
\begin{equation}\label{neq34}
\begin{aligned}
       \int_S^T \int_0^1 \phi (u-u_*) q(x)u_t \, dx  \, dt&= \frac{1}{2}\int_S^T \int_0^1 \big( \phi q(x) (u-u_*)^2\big)_t \, dx  \, dt\\
       &= \frac{1}{2}\left[ \int_0^1 \phi q(x) (u-u_*)^2 \, dx  \right]_S^T.
       \end{aligned}
\end{equation}
Putting together \eqref{lemmaS0:eq37},\eqref{lemmaS0:eq38} \eqref{lemmaS0:eq39}, \eqref{neq33}, \eqref{eeq37} and \eqref{neq34} gives
\begin{equation}\label{neweq40}
\begin{aligned}
\int_S^T \int_0^1 \phi (u_t^2 +au_x^2) \, dx  \, dt &+ \frac{a(1)\alpha_2}{\beta_1} \int_S^T\eta_2^2(t) \, dt= 2\int_S^T \int_0^1 \phi u_t^2 \, dx  \, dt- \left[\int_0^1 \phi (u-u_*) u_t \, dx \right]_S^T\\
& + \frac{1}{2}\int_S^T \int_0^1 (a\phi_{x})_x (u-u_*)^2 \, dx  \, dt 
-\frac{1}{2} \left[\int_0^1 \phi q(x) (u-u_*)^2 \, dx  \right]_S^T   \\ 
&   +\frac{a(1)}{\beta_1} \int_S^T \eta_1^2(t)\, dt- \left[\frac{a(1)\alpha_1}{2 \beta_1} \eta_2^2(t)+ \frac{a(1)}{\beta_1}\eta_1(t) \eta_2(t)\right]_S^T.
\end{aligned}
\end{equation}
Since $\phi=1$ on $Q_1 \cap (0,1)$ and $\phi \geq 0 $ on $(0,1)$, we have
\begin{equation*}
    \underbrace{\int_S^T \int_{Q_1 \cap (0,1)} (u_t^2 +au_x^2) \, dx  \, dt}_{\textbf{S}_0} \leq \int_S^T \int_0^1 \phi(u_t^2 +au_x^2)\, dx  \, dt
\end{equation*}
Hence, we have
\begin{equation}\label{eq42}
\begin{aligned}
     \mathbf{S}_0 + \frac{a(1)\alpha_2}{\beta_1} \int_S^T\eta_2^2(t) \, dt \leq & 2 \underbrace{ \int_S^T \int_0^1 \phi u_t^2 \, dx  \, dt}_{\textbf{T}_4}+ \underbrace{\left| \left[\int_0^1  (u-u_*) u_t \, dx \right]_S^T \right|}_{\textbf{T}_5}\\
     &+\frac{1}{2}\underbrace{\left|\int_S^T \int_0^1 (a\phi_{x})_x (u-u_*)^2 \, dx  \, dt\right|}_{\textbf{T}_6} 
      +\frac{1}{2}\underbrace{ \left| \left[\int_0^1 \phi q(x) (u-u_*)^2 \, dx  \right]_S^T \right|}_{\textbf{T}_7} 
     \\
     &+ \frac{a(1)}{\beta_1}\underbrace{ \int_S^T \eta_1^2(t)\, dt }_{\textbf{T}_8} +\underbrace{\left| \left[\frac{a(1)\alpha_1}{2 \beta_1} \eta_2^2(t)+ \frac{a(1)}{\beta_1}\eta_1(t) \eta_2(t)\right]_S^T \right|}_{\textbf{T}_9 }.    
\end{aligned}
\end{equation}
We estimate now the terms on the right-hand side of \eqref{eq42}.
Using the fact that $\phi \leq C q$ on $[0,1]$ and \eqref{ELocdot}, we have
\begin{equation}\label{eq43}
    \textbf{T}_4 \leq  C\int_S^T \int_0^1 q u_t^2 \, dx  \, dt  \leq C E(S).
\end{equation}
To estimate $\textbf{T}_5$, we note that for every $t\geq 0$ and $x \in [0,1]$ the following inequality holds
\begin{equation}\label{umoinsustar}
\begin{aligned}
         |u-u_*|^2 &\leq 2|u-u(t,1)|^2+2|u(t,1)-u_*|^2\\
         &\leq 2|u-u(t,1)|^2+2\eta_2^2(t),
\end{aligned}
\end{equation}
integrating over $[0,1]$ with respect to $x$ and using Poincaré's inequality, give
\begin{equation} \label{q36}
  \int_0^1 |u-u_*|^2dx \leq 2\int_0^1 u_x^2 \, dx + 2 \eta_2^2(t).
\end{equation}
Then, using Young's inequality and \eqref{q36} we obtain that
\begin{equation*}
\begin{aligned}
    \left| \int_0^1 (u-u_*) u_t \, dx \right| &\leq \frac{1}{2} \int_0^1 u_t^2 \, dx +\frac{1}{2} \int_0^1 |u-u_*|^2dx\\
    &\leq C\int_0^1 (u_t^2 +au_x^2 ) \, dx +\eta_2^2(t) \leq CE(t).
    \end{aligned}
    \end{equation*}
Then it follows
    \begin{equation}\label{eq44}
      \textbf{T}_5 =\left| \left[\int_0^1  (u-u_*) u_t \, dx \right]_S^T \right|\leq  2E(S)+ 2E(T) \leq 4E(S).
    \end{equation}
For $ \textbf{T}_6$, we have
\begin{equation}\label{eq45}
\begin{aligned}    
    \textbf{T}_6 &= \left| \int_S^T \int_0^1 (a\phi_{x})_x |u-u_*|^2 \, dx  \, dt\right| \\
    &\leq  \n{(a\phi_{x})_x} \underbrace{\int_S^T \int_{Q_2 \cap (0,1)} |u-u_*|^2 \, dx  \, dt}_{\textbf{S}_1}.    \end{aligned}
\end{equation}
To estimate $\textbf{T}_7$, we use  \eqref{umoinsustar}  and Poincaré's inequality, to obtain that  
\begin{equation}\label{eq46}
\begin{aligned}
   \int_0^1 \phi q(x) |u-u_*|^2 \, dx  \leq & 2\bar{q}\int_0^1 |u-u(t,1)|^2 \, dx + 2\bar{q} \eta_2^2(t)\\
    &\leq 2\bar{q}\int_0^1 u_x^2 \, dx + 2\bar{q} \eta_2^2(t)\\
    & \leq C E(t).
\end{aligned}
\end{equation}
Then we conclude
\begin{equation}\label{qq41}
    \textbf{T}_7 =\left| \left[\int_0^1 \phi q(x) (u-u_*)^2 \, dx  \right]_S^T \right|\leq  C\big( E(S)+ E(T) \big) \leq CE(S).
\end{equation}
As for $\textbf{T}_8$, we use \eqref{ELocdot} and the fact that $E$ is a non-increasing function of $t$ to find that
\begin{equation}\label{T_8}
    \textbf{T}_8=\int_S^T \eta_1^2(t)\, dt \leq \frac{1}{\widetilde{\alpha}_1}\int_S^T (-\dot{E}(t)) \, dt \leq \frac{1}{\widetilde{\alpha}_1} \big( E(S)+ E(T) \big) \leq \frac{2}{\widetilde{\alpha}_1}E(S).
\end{equation}
Finally, using Young's inequality and the fact that $E$ is a non-increasing function of $t$ give
\begin{equation}\label{qq42}
   \textbf{T}_9= \left|\left[\frac{a(1)\alpha_1}{2 \beta_1} \eta_2^2(t)+ \frac{a(1)}{\beta_1}\eta_1(t) \eta_2(t)\right]_S^T \right| \leq  C \big( E(S)+ E(T) \big) \leq CE(S).
\end{equation}
Gathering \eqref{eq42}, \eqref{eq43}
\eqref{eq44}, \eqref{eq45}, \eqref{qq41}, \eqref{T_8} and \eqref{qq42} we obtain that \eqref{eq36}.
\end{proof}

From Lemma \ref{lem:S0} and Lemma \ref{lemmaS0} the following inequality follows
\begin{equation} \label{atthispoint}
\int_S^T E(t) \, dt \leq C\underbrace{\int_S^T \int_{Q_2 \cap(0,1)}|u-u_*|^2 \, dx  \, dt}_{\mathbf{S}_1}+CE(S), \quad \forall \,  0\leq S \leq T.
\end{equation}
Next, we estimate $\mathbf{S}_1$. For this, we use a special multiplier which is concentrated on the part $Q_2 \cap (0,1)$ by the function $\varphi$. 

Let us define the function $v$
as the solution of
\begin{equation}\label{aeequ_v}
\left\{ \begin{array}{l}
(av_{x})_x=\varphi (u-u_*), \quad \forall x \in (0, 1) \\
v_x(t,0)=v_x(t,1)+\frac{\alpha_2}{\beta_1}v(t,1)=0.
\end{array}\right.
\end{equation}
\begin{lemma} For $v$ as defined in \eqref{aeequ_v}, we have the following estimates:
\begin{equation}\label{v1andv0}
|v(t, 0)|^2 +|v(t, 1)|^2 \leq  C E(t).
\end{equation} 
\begin{equation}\label{aelemma_vL}
\int_0^1|v|^2 \, dx  \leq CE(t).
\end{equation}

\begin{equation}\label{vt1andvt0}
|v_t(t, 0)|^2 +|v_t(t, 1)|^2 \leq  C \int_0^1 q|u_t|^2 \, dx.
\end{equation} 

\begin{equation}\label{aelemma_vtL}
\int_0^1\left|v_t\right|^2 \, dx \leq C\int_0^1 q \left|u_t\right|^2 \, dx.
\end{equation}

\end{lemma}

\begin{proof}
 From the definition of $v$, one gets
\begin{equation}\label{defv}
    v(t, x)=\int_0^x \Big(\int_s^x \frac{d\tau}{a(\tau)} \Big) \varphi (u-u_*) d s -\int_0^1 \Big( \frac{\beta_1}{a(1)\alpha_1} +\int_s^1 \frac{d\tau}{a(\tau)} \Big) \varphi (u-u_*) d s, \quad \forall x \in[0,1] .
\end{equation}
Using Cauchy-Schwarz inequality, we deduce
\begin{equation}\label{anew51}
|v(t, x)|^2 \leq  C \int_0^1\varphi|u-u_*|^2 \, dx, \quad \forall x \in[0,1].
\end{equation}
Then we use \eqref{umoinsustar} and Poincaré's inequality to obtain that 
\begin{equation}\label{anew52}
|v(t, x)|^2 \leq C \int_0^1 \varphi|u-u_*|^2 \, dx \leq C\Big( \int_0^1 \varphi \left|u_x\right|^2 \, dx+ \eta_2^2(t) \Big) \leq CE(t), \quad \forall x \in [0,1].
\end{equation}
yielding \eqref{v1andv0} and also \eqref{aelemma_vL} after integrating over $[0,1]$ with respect to $x$.\\
Similarly, one has
\begin{equation}\label{defvt}
v_t(t, x)=\int_0^x \Big(\int_s^x \frac{d\tau}{a(\tau)} \Big) \varphi u_t d s -\int_0^1 \Big( \frac{\beta_1}{a(1)\alpha_1} +\int_s^1 \frac{d\tau}{a(\tau)} \Big) \varphi u_t d s.
\end{equation}
Using Cauchy-Schwarz inequality and the fact that $\varphi \leq C q $, we infer that
\begin{equation}
|v_t(t, x)|^2 \leq  C \int_0^1 \varphi|u_t|^2 \, dx\leq  C \int_0^1 q|u_t|^2 \, dx, \quad \forall x \in[0,1]
\end{equation} 
and then \eqref{vt1andvt0} and \eqref{aelemma_vtL} follow. 
\end{proof}

\begin{lemma}\label{lemma16}
Assume  (A1), (A2), and (A3). Then there exists a positive constant $C>0$ such that strong solutions of system \eqref{sys_main_int} satisfy 
\begin{equation}\label{aeS1}
\underbrace{\int_S^T \int_{Q_2 \cap(0,1)}|u-u_*|^2 \, dx  \, dt}_{\textbf{S}_1} \leq C \delta_5  \int_S^T E(t) \, dt+C(\frac{1}{\delta_5}+1) E(S),
\end{equation}
for every $0\leq S \leq T$ and every $\delta_5>0$.
\end{lemma}
\begin{proof}
We multiply the first equation of (\ref{sys_main_int}) by $v$
\begin{equation}\label{aeeq61}
\int_S^T \int_0^1 v\left(u_{tt} -(au_{x})_x +q(x)u_t\right) \, dx  \, dt=0.
\end{equation}
First, an integration by part with respect to $t$ gives
\begin{equation}\label{aeeq62}
 \int_S^T \int_0^1v u_{tt}  \, dx  \, dt=-\int_S^T\int_0^1  v_t u_t  \, dx \, dt+\left[\int_0^1 v u_t \, dx\right]_S^T.
\end{equation}
Then, integrating by parts with respect to $x$ two times and using \eqref{aeequ_v} yields
\begin{equation}\label{aeeq63}
\begin{aligned} 
-\int_S^T\int_0^1 v (au_{x})_x \, dx &=\int_S^T\int_0^1 v_x au_x \, dx-\int_S^T [ v a u_x]_0^1 \, dt \\
&=-\int_S^T\int_0^1 (av_{x})_x (u-u_*) \, dx+\int_S^T [a v_x (u-u_*)]_0^1dt -\int_S^T [v au_x]_0^1 \, dt \\
&=\underbrace{\int_S^T a(1)v_x(t,1)\eta_2(t)\, dt  - \int_S^T a(1)v(t,1) u_x(t,1) \, dt+\int_S^T   a(0) v(t,0) u_x(t,0) \, dt}_{\textbf{I}}\\
&\quad -\int_S^T\int_0^1 \varphi |u-u_*|^2 \, dx.
\end{aligned}
\end{equation}
We simplify the three integrals  of $\textbf{I}$  in \eqref{aeeq63}.
Multiplying \eqref{eq:xi1} by $\frac{a(1)}{\beta_1} v(t,1)$ and \eqref{eq:xi0} by $\frac{a(0) }{\mu_1} v(t,0)$ then integrating over $[S,T]$ with respect to $t$ gives
\begin{equation}\label{qq57}
\begin{aligned}
     -\int_S^T a (1)v(t,1)u_x(t,1) \, dt &= \frac{a(1)}{\beta_1} \int_S^T \dot{\eta}_1(t) v(t,1) \, dt + \frac{a(1)\alpha_1}{\beta_1}\int_S^T \eta_1 v(t,1) \, dt\\ &+ \frac{a(1)\alpha_2}{\beta_1}\int_S^T \eta_2(t) v(t,1) \, dt,
\end{aligned}   
\end{equation}
\begin{equation}\label{qq58}
    \int_S^T  a(0) v(t,0) u_x(t,0) \, dt = \frac{a(0)}{\mu_1} \int_S^T \dot{\xi}_1 v(t,0) \, dt + \frac{a(0)\gamma_1}{\mu_1} \int_S^T \zeta_1 v(t,0) \, dt.
\end{equation}
Summing \eqref{qq57}, \eqref{qq58} and $\int_S^T a(1) v_x(t,1)\eta_2(t)\, dt$, then using the boundary conditions in \eqref{aeequ_v} we find that
\begin{equation}
\begin{aligned}
  \textbf{I}=& \frac{a(1)}{\beta_1}\int_S^T \dot{\eta}_1(t) v(t,1)\, dt +\frac{a(0)}{\mu_1}\int_S^T \dot{\xi}_1 v(t,0) \, dt\\ &+\frac{a(1)\alpha_1}{\beta_1} \int_S^T \eta_1(t) v(t,1)\, dt + \frac{a(0)\gamma_1}{\mu_1}\int_S^T \zeta_1 v(t,0)\, dt.
\end{aligned}
\end{equation}
An integration by parts with respect to $t$ gives
\begin{equation}\label{eq:59}
\begin{aligned}
     \textbf{I}= &-\frac{a(1)}{\beta_1}\int_S^T \eta_1(t) v_t(t,1)\, dt -\frac{a(0)}{\mu_1}\int_S^T \zeta_1(t) v_t(t,0) \, dt\\ &+\frac{a(1)\alpha_1}{\beta_1} \int_S^T \eta_1(t) v(t,1)\, dt
     + \frac{a(0)\gamma_1}{\mu_1}\int_S^T \zeta_1 (t)v(t,0)\, dt\\
     &+ \left[ \frac{a(1)}{\beta_1} \eta_1(t) v(t,1) + \frac{a(0)}{\mu_1} \zeta_1(t) v(t,0)\right]_S^T.
\end{aligned}   
\end{equation}
Combining \eqref{aeeq61}, \eqref{aeeq62}, \eqref{aeeq63} and \eqref{eq:59} we obtain that
\begin{equation}\label{qq56}
\begin{aligned}
\int_S^T \int_0^1 \varphi|u-u_*|^2 \, dx  \, dt=&-  \int_S^T\int_0^1 v_t u_t  d xdt+\left[\int_0^1 v u_t \, dx\right]_S^T+\int_S^T \int_0^1  v q(x)u_t \, dx  \, dt\\
& -\frac{a(1)}{\beta_1}\int_S^T \eta_1(t) v_t(t,1)\, dt -\frac{a(0)}{\mu_1}\int_S^T \zeta_1(t) v_t(t,0) \, dt\\
&+ \frac{a(1)\alpha_1}{\beta_1} \int_S^T \eta_1(t) v(t,1) \, dt+ \frac{a(0)\gamma_1}{\mu_1}\int_S^T \zeta_1 (t)v(t,0)\, dt\\
     &+\left[ \frac{a(1)}{\beta_1} \eta_1(t) v(t,1) + \frac{a(0)}{\mu_1} \zeta_1(t) v(t,0)\right]_S^T.
\end{aligned}
\end{equation}
Since $\varphi=1$ on $Q_2 \cap (0,1)$ and $\varphi \geq 0$ on $(0,1)$, we have
\begin{equation}
    \underbrace{\int_S^T \int_{Q_2 \cap(0,1)}|u-u_*|^2 \, dx  \, dt}_{\mathbf{S}_1} \leq \int_S^T \int_0^1 \varphi|u-u_*|^2 \, dx  \, dt.
\end{equation}
Hence we have that
\begin{equation}\label{aeS1aveclesv}
\begin{aligned}
\mathbf{S}_1 \leq & \underbrace{\int_S^T \int_0^1\left|v_t\right||u_t| \, dx  \, dt}_{\mathbf{V}_1}+\underbrace{\left|\left[\int_0^1 vu_t \, dx\right]_S^T\right|}_{\mathbf{V}_2}+\underbrace{\int_S^T \int_0^1|v q(x)u_t| \, dx  \, dt}_{\mathbf{V}_3}\\
&+\underbrace{\left|\frac{a(1)}{\beta_1}\int_S^T \eta_1(t) v_t(t,1)\, dt +\frac{a(0)}{\mu_1}\int_S^T \zeta_1(t) v_t(t,0) \, dt\right|}_{\textbf{I}_1}\\
&+ \underbrace{\left|\frac{a(1)\alpha_1}{\beta_1} \int_S^T \eta_1(t) v(t,1) \, dt+ \frac{a(0)\gamma_1}{\mu_1}\int_S^T \zeta_1 (t)v(t,0)\, dt\right|}_{\textbf{I}_2}\\
     &+ \underbrace{\left|\left[ \frac{a(1)}{\beta_1} \eta_1(t) v(t,1) + \frac{a(0)}{\mu_1} \zeta_1(t) v(t,0)\right]_S^T\right|}_{\textbf{I}_3}.
\end{aligned}
\end{equation}
We start by estimating $\mathbf{V}_1$. 
Using Young's inequality with $\delta_2>0$ we have
\begin{equation}\label{aeequ60}
\mathbf{V}_1 = \int_S^T \int_0^1 |v_t| |u_t| \, dx  \, dt \leq  \frac{1}{2\delta_2} \int_S^T \int_0^1\left|v_t\right|^2 \, dx  \, dt+ \frac{\delta_2}{2} \int_S^T \int_0^1|u_t|^2 \, dx  \, dt.
\end{equation}
From \eqref{aelemma_vtL} and \eqref{ELocdot} we find that 
\begin{equation}\label{aeequ62}
\int_S^T \int_0^1\left|v_t\right|^2 \, dx  \, dt\leq  C\int_S^T \int_0^1 q(x)|u_t|^2  \, dx  \, dt  \leq CE(S).
\end{equation}
Combining \eqref{aeequ60} and \eqref{aeequ62} we obtain that
\begin{equation}\label{av1}
\mathbf{V}_1  \leq \frac{\delta_2}{2} \int_S^T E(t) \, dt+ \frac{C}{ \delta_2} E(S).
\end{equation}
We now estimate $\mathbf{V}_2$. For a fixed $t \in[S, T]$, using Young's inequality and (\ref{aelemma_vL}) gives
\begin{equation}
\left|\int_0^1 vu_t \, dx\right| \leq \int_0^1|v||u_t| \, dx \leq \frac{1}{2}  \int_0^1|v|^2 \, dx+\frac{1}{2} \int_0^1|u_t|^2d x \leq C E(t).
\end{equation}
Since $E$ is a non-increasing function of $t$, we get
\begin{equation}\label{av2}
\mathbf{V}_2 =\left|\left[\int_0^1 vu_t \, dx\right]_S^T\right| \leq C \big( E(S)+ E(T) \big) \leq CE(S).
\end{equation}
For $\mathbf{V}_3$, we use Young’s inequality with $\delta_3>0$ and obtain that
\begin{equation} \label{aeeq72}
\mathbf{V}_3 \leq \frac{\delta_3}{2} \int_S^T \int_0^1 |v|^2 \, dx  \, dt + \frac1{2\delta_3} \int_S^T \int_0^1 q(x) u_t^2 \, dx  \, dt,
\end{equation}
then using \eqref{aelemma_vL} and \eqref{aeequ62} gives
\begin{equation}\label{av3}
\mathbf{V}_3 \leq C \delta_3 \int_S^TE(t)\, dt + \frac{C}{\delta_3} E(S).
\end{equation}
We now estimate $\textbf{I}_1$. Using Young's inequality with $\delta_4>0$ and \eqref{vt1andvt0} we find that
\begin{equation}\label{qq71}
\begin{aligned}
    \textbf{I}_1 &\leq \left|\frac{a(1)}{\beta_1}\int_S^T \eta_1(t) v_t(t,1)\, dt +\frac{a(0)}{\mu_1}\int_S^T \zeta_1(t) v_t(t,0) \, dt\right|\\
    &\leq C \delta_4\int_S^T (\eta_1^2(t)+\zeta_1^2(t)) \, dt + \frac{C}{\delta_4} \int_S^T\Big( |v_t(t, 0)|^2 +|v_t(t, 1)|^2 \Big)\, dt\\
    &\leq C \delta_4\int_S^T (\eta_1^2(t)+\zeta_1^2(t)) \, dt + \frac{C}{\delta_4} \int_S^T\int_0^1 qu_t^2 d x dt.
    \end{aligned}
   \end{equation}

For $\textbf{I}_2$, we use Young's inequality with $\delta_5>0$ and  \eqref{v1andv0}, to find that
\begin{equation}\label{qq75}
    \begin{aligned}
       \textbf{I}_2 & \leq \left|\frac{a(1)\alpha_1}{\beta_1} \int_S^T \eta_1(t) v(t,1) \, dt+ \frac{a(0)\gamma_1}{\mu_1}\int_S^T \zeta_1 (t)v(t,0)\, dt\right|\\
       &\leq  C \delta_5 \int_S^T \Big(|v(t, 0)|^2 +|v(t, 1)|^2 \Big)\, dt+\frac{C}{\delta_5} \int_S^T ( \eta_1^2(t)+ \zeta_1^2(t)) \, dt \\
       &\leq  C \delta_5 \int_S^T E(t)\, dt + \frac{C}{\delta_5} \int_S^T ( \eta_1^2(t)+ \zeta_1^2(t)) \, dt.
    \end{aligned}
\end{equation}

For $\textbf{I}_3$, we rely on  Young's inequality and \eqref{v1andv0} to obtain that
\begin{equation*}
\begin{aligned}
      \left| \frac{a(1)}{\beta_1} \eta_1(t) v(t,1) + \frac{a(0)}{\mu_1} \zeta_1 v(t,0)\right| &\leq C( \eta_1^2(t)+ \zeta_1^2(t)) + C \Big( |v(t, 0)|^2 +|v(t, 1)|^2\Big)\\
      &\leq C ( \eta_1^2(t)+ \zeta_1^2(t)) + C E(t)\\
      &\leq CE(t),
\end{aligned} 
\end{equation*}
Then since $E$ is a non-increasing function of $t$ we obtain that  
\begin{equation}\label{qq76}
   \textbf{I}_3 \leq CE(S).
\end{equation}
Gathering \eqref{aeS1aveclesv}, \eqref{av1}, \eqref{av2}, \eqref{av3},\eqref{qq71}, \eqref{qq75} and \eqref{qq76} we obtain that
\begin{equation}\label{S1beforequ_t}
\begin{aligned}
      \textbf{S}_1\leq & C \left(\delta_4 + \frac{1}{\delta_5} \right) \int_S^T (\eta_1^2(t)+\zeta_1^2(t)) \, dt +C \Big(  \delta_2+\delta_3 +\delta_5 \Big)\int_S^T E(t)\, dt\\
      &+\frac{C}{\delta_4}\int_S^T \int_0^1 q u_t^2 \, dx  \, dt + C\Big(\frac{1}{\delta_2} + \frac{1}{\delta_3}+ 1\Big)E_"(S). 
\end{aligned}
   \end{equation}
We now give the following inequality that we obtain from \eqref{ELocdot} and that plays a crucial role in finishing the estimate of $\textbf{S}_1$.
\begin{equation}\label{qu_testimate}
    \begin{aligned}
\int_S^T\int_0^1 q|u_t|^2 \, dx \, dt&=  \int_S^T \Big(-\dot{E}_2(t) - \widetilde{\alpha}_1 \eta_1^2(t)-\widetilde{\gamma}_1\zeta_1^2(t)\Big) \, dt\\
&\leq E(S) -E(T) - \widetilde{C}\int_S^T \Big(\eta_1^2(t) +\zeta_1^2(t)\Big) \, dt \\
&\leq E(S)- \widetilde{C}\int_S^T \Big(\eta_1^2(t) +\zeta_1^2(t)\Big) \, dt ,  \end{aligned}
\end{equation}
 where $\widetilde{C}=\min \{\widetilde{\alpha}_1,  \widetilde{\gamma}_1\}$ such that $\widetilde{\alpha}_1$ and  $\widetilde{\gamma}_1$ are given in \eqref{ELocdot}.
Putting \eqref{qu_testimate} in \eqref{S1beforequ_t} we obtain that

\begin{equation}\label{S1afterqu_t}
\begin{aligned}
      \textbf{S}_1\leq & C\left( \delta_4 + \frac{1}{\delta_5} - \frac{\widetilde{C}}{\delta_4}  \right) \int_S^T (\eta_1^2(t)+\zeta_1^2(t)) \, dt +C \Big(\delta_2+\delta_3 +\delta_5 \Big)\int_S^T E(t)\, dt\\
      & +C\Big(\frac{1}{\delta_2} + \frac{1}{\delta_3}+\frac{1}{\delta_4}+ 1\Big)E(S). 
\end{aligned}
   \end{equation}
Choosing $\delta_4=\frac{\widetilde{C}} {2}\delta_5 $  and $\delta_2=\delta_3=\delta_5$ yields
\begin{equation}\label{eq:76s1}
\begin{aligned}
       \textbf{S}_1\leq & C\left( \delta_5 
 - \frac{1}{\delta_5}\right)\int_S^T (\eta_1^2(t)+\zeta_1^2(t)) \, dt +C \delta_5 \int_S^T E(t)\, dt\\
      &+ C\Big( \frac{1}{\delta_5} + 1 \Big)E(S).  
\end{aligned}
\end{equation}
Then, up to keeping $\delta_5$ sufficiently small we obtain that 
\begin{equation}
       \textbf{S}_1\leq C  \delta_5 \int_S^T E(t)\, dt + C\Big(\frac{1}{\delta_5} + 1 \Big)E(S).
\end{equation}
This yields \eqref{aeS1}.
\end{proof}

From \eqref{atthispoint} and Lemma \ref{lemma16} we now have that
\begin{equation}
\int_S^T E(t)  \, dt  \leq C \delta_5  \int_S^T E(t) \, dt+C \big( \frac{1}{\delta_5} +1\big)E(S). 
\end{equation}
Up to choosing $\delta_5$ sufficiently small, we conclude
\begin{equation}
     \int_S^T E(t) \, dt \leq C E(S).
\end{equation}
Using \cite[Theorem 8.1]{komornik1994exact}, we obtain \eqref{th:mainth-cl1}.  

\medskip 
The proof is complete. 
\qed

\subsection{Frequency domain approach}\label{sec:frequency}

In this subsection, we give an alternative proof to the stability result from the previous subsection. 
Recall that $u_*$ is defined in \eqref{u_*} and \eqref{prop-u_*} holds. Note that if $(u,\eta_1, \eta_2, \zeta_1)$ is a weak solution then $(u + c,  \eta_1, \eta_2, \zeta_1)$ with $c$ being a constant is also a weak solution. By appropriately choosing the constant $c$, we get
\be
u_* = 0. \label{u_star_equ_zero}
\ee
Therefore, without loss of generality, we assume that \eqref{u_star_equ_zero} holds true in the rest of this subsection. Under this assumption, one can now involve the spectral theory of strongly continuous semigroups to conclude \eqref{th:mainth-cl1}. The details are below. 

We first introduce the energy space associated with problem \eqref{sys_main_int}.
\begin{equation}
    \mathcal{H}:= H^1\times L^2\times \mathbb{C}^3,
\end{equation}
where $H^j:=H^j([0,1],\mathbb{C})$, $j=1$ or $j=2$, and $L^2:=L^2([0,1],\mathbb{C})$ endowed with the natural inner product.
Consider the following hyperplane of $\cH$: 
\begin{equation}\label{hyperplan}
       \mathcal{H}_0:= \Big\{ \bU= (\mathbf{u},\mathbf{v}, \boldsymbol{\eta}_1, \boldsymbol{\eta}_2, \boldsymbol{\zeta}_1) \in \mathcal{H}: \mathbf{u}(1)=\boldsymbol{\eta}_2 \Big\}.
\end{equation}

We endow $\mathcal{H}_0$ with the following inner product
\begin{equation}\label{innerproduct}
    \langle \mathbf{U}, \widetilde{\mathbf{U}} \rangle_{\mathcal{H}_0}= \int_0^1 a \mathbf{u}_x \bar{\widetilde{\mathbf{u}}}_x \, dx + \int_0^1 \mathbf{v} \bar{\widetilde{\mathbf{v}}} \, dx + 
         \frac{a(1)}{p\beta_1} \boldsymbol{\eta}_1 \bar{\widetilde{\boldsymbol{\eta}}}_1 +\frac{a(1) \alpha_2}{p\beta_1}   \bar{\widetilde{\boldsymbol{\eta}}}_2 + \frac{a(0)}{p\mu_1}  \boldsymbol{\zeta}_1 \bar{\widetilde{\boldsymbol{\zeta}}}_1,
\end{equation}
where 
$\mathbf{U}=(\mathbf{u},\mathbf{v}, \boldsymbol{\eta}_1, \boldsymbol{\eta}_2, \boldsymbol{\zeta}_1)^T$ and $\widetilde{\mathbf U}= (\widetilde{\mathbf{u}},\widetilde{\mathbf{v}}, \widetilde{\boldsymbol{\eta}}_1, \widetilde{\boldsymbol{\eta}}_2, \widetilde{\boldsymbol{\zeta}}_1)^T$ belong to $\cH_0$. One can check that the inner products $\langle \cdot, \cdot \rangle_{\cH_0}$ and $\langle \cdot, \cdot \rangle_{\cH}$ are equivalent in $\cH_0$,
since, for $\bU = (\mathbf{u},\mathbf{v}, \boldsymbol{\eta}_1, \boldsymbol{\eta}_2, \boldsymbol{\zeta}_1) \in \cH_0$,  
\[
\int_0^1|\mathbf{u}|^2 \, dx\leq 2\big(|\mathbf{u}(1)|^2+\int_0^1|\mathbf{u}_x|^2dx\big)=2\big(|\boldsymbol{\eta}_2|^2+\int_0^1|\mathbf{u}_x|^2dx\big)\leq C\|\mathbf{U}\|_{\mathcal{H}_0}^2.
\]

Consider the following subspace of ${\mathcal{H}_0}$ given by
\begin{equation}
    \D:=\Big\{\bold{U} = (\mathbf{u},\mathbf{v}, \boldsymbol{\eta}_1, \boldsymbol{\eta}_2, \boldsymbol{\zeta}_1)^T \in \mathcal{H}_0\mid \ \bold{u}\in H^2(0, 1),\ \bold{v} \in H^1(0,1),\ \bold{v}(1)=\boldsymbol{\eta}_1,\ \bold{v}(0)=\boldsymbol{\zeta}_1 \Big\}.
\end{equation}
Note that, since $a \in W^{1, \infty}(0, 1)$ then $au_x \in H^1(0, 1)$ if $u \in H^2(0, 1)$. 
Equip $\D$ with the following scalar product
$$
\langle \bU, \widetilde \bU \rangle_{\D} = \langle \bU, \widetilde \bU \rangle_{\cH_0} + \int_0^1 \bu_{xx} \bar \tbu_{xx} \, dx +  \int_0^1 \bv _x \bar \tbv_x \, dx \quad \mbox{ for } \bU, \tbU \in \D, 
$$
which endows $\D$ with a structure of Hilbert space. Set 
 \begin{equation}
    \mathcal{A}\bold{U} := \left( \begin{array}{c} 
    \bold{v} \\
    (a\bold{u}_{x})_x-q\bold{v} \\
    -\alpha_1 \boldsymbol{\eta}_1 -\alpha_2 \boldsymbol{\eta}_2 -\beta_1 \bold{u}_x(1)\\
    \boldsymbol{\eta}_1\\
    -\gamma_1 \boldsymbol{\zeta}_1 +\mu_1 \bold{u}_x(0) 
    \end{array} \right) \quad \mbox{ for } \, \bold U \in \D. 
\end{equation}
By the Sobolev embedding, one gets that the embedding $\D$ into $\cH_0$ is compact.
If we denote $u_t=v$, $U=(u,v, \eta_1, \eta_2, \zeta_1)^T$ and $U_0=(u_0, u_1, \eta_{1,0}, \eta_{2,0}, \zeta_{1,0})^T$, system \eqref{sys_main_int} is then written as the abstract Cauchy problem given by
\begin{equation}\label{compact_sys}
    \left\{ \begin{array}{l}
         U_t=\mathcal{A}U,   \\
         U(0)=U_0 \in \mathcal{H}_0,
    \end{array} \right.
\end{equation}

In what follows, we denote 
\be \label{def-c}
c_1 = a(1)/\beta_1, \quad c_2 = a(1) \alpha_2 / \beta_1, \quad \mbox{ and } \quad c_{3} = a(0)/\mu_1. 
\ee

The following elementary result is repeatedly used in this section. 
\begin{lemma} \label{lem-A} Assume (A1), (A2), and (A3). For $\bU \in \D$, we have 
\be \label{lem-A-cl1}
\langle - \cA \bU, \bU \rangle_{\cH_0} =  -  \int_0^1 a \bv_x \bar  \bu_x + \int_0^1 a \bu_x \bar \bv_x  + \int_0^1 q |\bv|^2 + c_1  \alpha_1 |\boeta_1|^2 + c_2 \boeta_2 \bar \boeta_1 - c_2 \boeta_1 \bar  \boeta_2  + c_3 \gamma_1 |\bzeta_1|^2. 
\ee
Consequently,  
\be \label{lem-A-cl2} 
\Re \langle - \cA \bU, \bU \rangle_{\cH_0} =  \int_0^1 q |\bv|^2 + c_1  \alpha_1 |\boeta_1|^2 +  c_3 \gamma_1 |\bzeta_1|^2.
\ee
\end{lemma}

\begin{proof} We have, for $\bU \in \D$,  by an integration by parts, 
\begin{multline*}
\langle - \cA \bU, \bU \rangle_{\cH_0} =  -  \int_0^1 a \bv_x \bar  \bu_x + \int_0^1 a \bu_x \bar \bv_x  + \int_0^1 q \bv \bar \bv - a (1) \bu_x (1) \bar \bv(1) + a(0) \bu_x(0) \bar \bv(0) \\[6pt]
+ c_1 \Big( \alpha_1 \boeta_1 + \alpha_2 \boeta_2 + \beta_1 \bu_x(1) \Big) \bar \boeta_1 - c_2  \boeta_1 \bar \boeta_2 - c_3 \Big( -\gamma_1 \bzeta_1 + \mu_1 \bu_x(0) \Big) \bar \bzeta_1.  
\end{multline*}
Using the fact $\bv(1) = \boeta_1$ and $\bv(0) = \bzeta_1$,  and \eqref{def-c}, after simplyfing the expression, we obtain \eqref{lem-A-cl1}. Assertion \eqref{lem-A-cl2} is a direct consequence of \eqref{lem-A-cl1}. 
\end{proof}
The following result deals with the existence and uniqueness of a solution to \eqref{compact_sys} via the semigroup theory.

\begin{theorem}\label{wellposedness}
Assume $(A1)$, $(A2)$, and $(A3)$ hold true. Then $\cA$ generates a strongly continuous semigroup $\big(S(t) \big)_{t \ge 0}$ of contractions. Consequently,  
\begin{itemize}
\item[i)] for  $U_0 \in \mathcal{H}_0$, there exists a unique (weak) solution $U \in C([0,+ \infty),\mathcal{H}_0)$ of problem \eqref{compact_sys}. 

\item[ii)] for $U_0 \in \mathcal{D}(\mathcal{A})$, then the unique solution $U \in C([0,+ \infty),\mathcal{H}_0)$ of problem \eqref{compact_sys} also satisfies
$$
U \in C([0,+\infty), \mathcal{D}(\mathcal{A})) \cap C^{1}([0,+\infty), \mathcal{H}_0).
$$
\end{itemize}
\end{theorem}

\begin{proof}
The proof of Theorem \ref{wellposedness} is based on Hille-Yosida Theorem and is standard. As a consequence of \eqref{lem-A-cl2} of \Cref{lem-A}, we derive that  
$\cA$ is dissipative on $\cH_0$ and, by applying Lax-Milgram's theorem,  
$\cA$ is maximally dissipative on $\cH_0$. The conclusion now follows from the Hille-Yosida theorem, see, e.g., \cite[Thm. 3.15 \& Prop. 6.2]{engel2000one} (see also \cite[Chapter 7]{brezis2011functional}).
\end{proof}

In the alternative proof of the exponential stability of system \eqref{sys_main_int}, we use
the following result.

\begin{theorem}{(Huang-Prüss, \cite{huang1985, pruss1984})}\label{HuangPruss} Let $\left(\mathbb{S}(t)\right)_{t \geq 0}$ be a strongly continuous semigroup of contractions on a Hilbert space $H$ and let $\mathbb{A}$ be its infinitesimal generator. Then $\left(\mathbb{S} (t) \right)_{t \geq 0}$  is exponentially stable if and only if
\begin{equation}\tag{$H_1$}\label{H1}
i \mathbb{R} \subset \rho(\mathbb{A}),
\end{equation}
and
\begin{equation}\tag{$H_2$}\label{H2}
\sup _{\lambda \in \mathbb{R}}\left\|(i \lambda I-\mathbb{A})^{-1}\right\|_{\mathcal{L}(H)}<\infty,
\end{equation}
where $\mathbb{A}$ is the infinitesimal generator of $\left(S(t)\right)_{t \geq 0}$ with resolvent set denoted $\rho(\mathbb{A})$. 
\end{theorem}

Applying \Cref{HuangPruss}, we can derive the following result. 

\begin{theorem}\label{expostability}
  Assume  $(A1)$, $(A2)$, and $(A3)$ hold true. Then the strongly continuous semigroup $\big(S(t) \big)_{t \ge 0}$ of contractions generated by $\cA$  is exponentially stable.
\end{theorem}

\Cref{expostability} is the main result of this subsection. It yields \eqref{th:mainth-cl1}, and \eqref{th:mainth-cl2} follows easily as shown in the previous subsection.  To be able to apply \Cref{HuangPruss}, we establish the following two results. Here is the first one.  

\begin{proposition}\label{kerilambda-A}
 Assume  $(A1)$, $(A2)$, and $(A3)$ hold true. Then $i \mathbb{R} \subset \rho(\mathcal{A})$.
\end{proposition}

\begin{proof} From the maximal dissipativity of $\cA$, we derive that $(I - \cA): \D \to \cH_0$ is continuously invertible. Using the Fredholm theory, it suffices to prove that 
 \begin{equation}\label{eq:NoeigenvalueiniR}
        \operatorname{Ker} (i \lambda I -\mathcal{A})=\{0\},\qquad \forall \, \lambda \in \mathbb{R}.
    \end{equation}
    
Let $\lambda\in\mathbb R$ and $U \in \operatorname{Ker} (i \lambda I -\mathcal{A})$. Then 
 \begin{equation}\label{eq:AUequalilambda}
     \mathcal{A} U= i \lambda U.
 \end{equation}    
This implies
    \begin{equation}
       i \lambda \|U\|_{\mathcal{H}_0}^2=  \langle \mathcal{A}U, U\rangle_{\mathcal{H}_0}. 
    \end{equation} 
Applying \Cref{lem-A}, we obtain 
\be
\int_0^1 q |v|^2 + c_1  \alpha_1 |\eta_1|^2 +  c_3 \gamma_1 |\zeta_1|^2 =0. 
\ee
We derive that
    \begin{equation}\label{kerilambda-A-eta1xi1=0}
        \eta_1=\zeta_1=0.
    \end{equation}
From \eqref{eq:AUequalilambda}, we obtain 
\begin{equation} \label{kerilambda-A-p1}
    \left\{ \begin{array}{rl}
     \lambda^2 u + (au_{x})_x &=0 \mbox{ in } (0, 1), \\
    u_x(1) &=0,\\
     u_x(0)&=0, 
           \end{array} \right.
\end{equation}
\begin{equation} \label{kerilambda-A-p2}
    i \lambda u -v =0 \mbox{ in } (0, 1), 
\end{equation}
and 
\begin{equation} \label{kerilambda-A-p3}
     \eta_2 =0. 
\end{equation}
Since $U \in \cH_0$,  we also have that $u(1) = 0$. Combining that with \eqref{kerilambda-A-p1}, one can see that, if we set $V=(u,u_x)$, then $V$ is solution of a Cauchy problem of the type $V_x=A(x)V$ with $V(x_0)=0$ for some $x_0 \in \omega$ and $x\mapsto A(x)$ bounded measurable map. By the Cauchy-Lipschitz theorem for ODEs, it follows that 
$V=0$ on $[0,1]$ and hence 
\be \label{kerilambda-A-p5}
u = 0 \mbox{ in } [0, 1]. 
\ee
This in turn implies, by \eqref{kerilambda-A-p2} that $v=0$ on $[0,1]$, which yields after considering \eqref{kerilambda-A-eta1xi1=0}, \eqref{kerilambda-A-p3}, and that
$U = 0$ on $[0,1]$.
The proof of the proposition is complete.
 \end{proof}

Proposition \ref{kerilambda-A} shows that hypothesis \eqref{H1} in  \Cref{HuangPruss} is satisfied. 
In the second result, we verify \eqref{H2} in  \Cref{HuangPruss}. The analysis is more delicate.

\begin{proposition} \label{pro-A}
Assume  $(A1)$, $(A2)$, and $(A3)$ hold true. Then     
\begin{equation}\label{resolventestimate}
\sup _{\lambda \in \mathbb{R}}\left\|(i \lambda I-\mathcal{A})^{-1}\right\|_{\mathcal{L}(\mathcal{H}_0)}<\infty.
\end{equation}
\end{proposition}
\begin{proof}
    We argue by contradiction. Suppose that \eqref{resolventestimate} does not hold. Then, there exist a sequence $(\lambda_n)_{n \in \mathbb{N}}$ of real numbers and a bounded sequence $(F_n)_{n \in \mathbb{N}} \in \mathcal{H}_0$ such that
\begin{equation}\label{eq:contradiction}
\|(i \lambda_n I-\mathcal{A})^{-1}F_n\|_{\mathcal{H}_0}\xrightarrow[n\to \infty]{} \infty.
\end{equation}

Moreover, one necessarily has that $|\lambda_n|\to\infty$. On the other hand, it is immediate to deduce from \eqref{eq:contradiction} that
there exists a sequence $(U_n)_{n\in \mathbb N}$ in $\mathcal{D}(\mathcal{A})$ such that
\begin{equation*}
    (i \lambda_n I-\mathcal{A})^{-1}F_n=U_n.
\end{equation*}
Normalizing the previous equation with respect to $(U_n)_{n\in \mathbb N}$ gives that
\begin{gather}
    \left|\lambda_n\right| \rightarrow+\infty, \label{lambdangoestoinf}\\
    \|U_n\|_{\mathcal{H}_0}=1, \label{normequalsone}\\
    (i \lambda_n I - \mathcal{A}_n) U_n = o(1) \text{ in } \mathcal{H}_0.\label{systozero}
\end{gather}
Denote  $U_n=(u_n,v_n, \eta_{1,n}, \eta_{2,n}, \zeta_{1,n})^T$. Taking the real part of the inner product of \eqref{systozero}
 with $U_n$ in $\mathcal{H}_0$ gives that
\begin{equation}\label{eq:RAUU=o}
    \Re \langle \mathcal{A} U_n, U_n\rangle_{\mathcal{H}_0} =o(1).
\end{equation}
Applying \Cref{lem-A}, we obtain 
\begin{equation}\label{Edotpetito}
     \int_0^1 q |v_n|^2 \, dx +  \frac{a(1) \alpha_1}{\beta_1}|\eta_{1,n}|^2 + \frac{a(0)\gamma_1}{\mu_1} |\zeta_{1,n}|^2 =o(1).
\end{equation}
From \eqref{systozero}, we have
\begin{subnumcases}{}
     i \lambda_n u_n- v_n = o(1) \mbox{ in } H^1(0, 1) , \label{eq25a}\\
    i \lambda_n v_n -(au_{nx})_x = o(1) \mbox{ in } L^2(0, 1),  \label{eq25b}\\
    i \lambda_n \eta_{1,n}  +\alpha_2 \eta_{2,n} +\beta_1 u_{nx}(1)=o(1) \mbox{ in } \mC, \label{eq25c}\\
   i \lambda_n \eta_{2,n}  = o(1) \mbox{ in } \mC,  \label{eq25d} \\
    i \lambda_n \zeta_{1,n}  - \mu_1 u_{nx}(0) =o(1) \mbox{ in } \mC. \label{eq25e}
\end{subnumcases}
It follows from \eqref{Edotpetito} that 
\begin{equation} \label{eq25c'} 
u_{nx}(1) = \lambda_n o(1) \quad \mbox{ and } \quad u_{nx} (0) = \lambda_n o(1). 
\end{equation}
From \eqref{eq25a} and \eqref{eq25b}, we derive that 
\be \label{eq-un}
\lambda_n^2 u_n + (a u_{n, x})_x = o(1) \mbox{ in } L^2(0, 1). 
\ee
From \eqref{eq25a}, it follows since $v_n(0)=\xi_{1,n}$ that
\begin{equation*}
    | i \lambda_nu_n(0) - \zeta_{1,n}| \leq \| i \lambda_nu_n-v_n\|_{\infty} \leq  C  \| i \lambda_nu_n-v_n\|_{H^1(0, 1)} =o(1). 
\end{equation*}
This yields, by \eqref{Edotpetito}, that
\begin{equation} \label{pro-A-p1}
|i \lambda_nu_n(0)| \leq o(1)+ |\zeta_{1,n}| \mathop{=}^{\eqref{Edotpetito}} o(1). 
\end{equation}
Combining \eqref{eq25d} and \eqref{pro-A-p1} yields 
\begin{equation} \label{pro-A-un01}
|\lambda_nu_n(0)| + |\lambda_nu_n(1)| =  o(1). 
\end{equation}

Let $h \in C^1([0, 1],\mR)$. Multiplying \eqref{eq-un} by $h \bar u_n$ and integrating by parts, we obtain after using \eqref{eq25c'} and \eqref{pro-A-p1}, we obtain 
\begin{equation}\label{eqh:110}
\int_0^1 h\big( a|u_{nx}|^2 - \lambda_n^2 |u_n|^2 \big)\, dx +\int_0^1 h_xa u_{nx}\bar{u}_n \, dx = o(1).
\end{equation}

After choosing $h=1$, we derive from \eqref{eqh:110} that 
\begin{equation}\label{unequivn}
 \int_0^1 a|u_{nx}|^2 \, dx-\int_0^1 \lambda_n^2 |v_n|^2 \, dx =o(1).
\end{equation}

Let $\varphi \in C^1([0, 1],\mR)$. Multiplying \eqref{eq-un} by $\varphi \bar u_{n, x}$, we have 
\be
\int_0^1 \Big( \lambda_n^2 \varphi u_n \bar u_{n, x} + (a u_{n, x} ) \varphi \bar u_{n, x} \Big) \, dx = \int_0^1 o(1)_{L^2(0, 1)} \varphi \bar u_{n, x} \, dx.   
\ee
We have 
\be \label{pro-A-p4}
\Re \left\{ \int_0^1 \lambda_n^2 \varphi u_n \bar u_{n, x} \, dx \right\} = \frac{1}{2} \int_0^1 \lambda_n^2 \varphi \partial_x  |u_{n}|^2 \, dx = - \frac{\lambda_n^2}{2} \int_0^1 |u_n|^2 \varphi_x + o(1)  
\ee
by \eqref{eq25c'} and \eqref{pro-A-un01}. 
 
We also have 
\begin{multline}
   \Re \int_0^1 (au_{nx})_x \varphi \bar{u}_{nx} \, dx = \Re \int_0^1 (au_{nx})_x \frac{\varphi}{a}  a \bar{u}_{nx} \, dx = \frac{1}{2}\int_0^1 \frac{\varphi}{a} (|au_{nx}|^2)_x\, dx
\\[6pt]
= - \frac{1}{2}\int_0^1 (\varphi/a)_x |au_{nx}|^2\, dx + \left[ (\varphi/a) | a u_{nx}|^2\right]_0^1.
\end{multline} 
Using \eqref{eq25c'}, we derive that 
\begin{equation}\label{pro-A-p5}
  \Re \left(  \int_0^1 (a u_{nx})_x\varphi \bar{u}_{nx} \, dx\right) = - \frac{1}{2}\int_0^1 (\varphi/a)_x |au_{nx}|^2\, dx  + o(1). 
\end{equation}
Combining \eqref{pro-A-p4}  and \eqref{pro-A-p5} and using the fact $\| U_n \|_{\cH_0} =1$, we obtain 
\be \label{pro-A-p6}
\frac{\lambda_n^2}{2} \int_0^1 \varphi_x |u_n|^2 + \frac{1}{2}\int_0^1 (\varphi/a)_x |au_{nx}|^2\, dx =  o(1).
\ee
Using \eqref{eqh:110} with $h = \varphi_x$, we obtain 
\begin{equation}
\int_0^1 \varphi_x \big( a|u_{nx}|^2 - \lambda_n^2 |u_n|^2 \big)\, dx +\int_0^1 \varphi_{xx} a u_{nx}\bar{u}_n \, dx = o(1).
\end{equation}
Since 
$$
\int_0^1 \varphi_{xx} a u_{nx}\bar{u}_n \, dx = o(1)
$$
by \eqref{unequivn} and the fact that $\|U_n \|_{\cH_0} =1$, it follows that 
\begin{equation} \label{pro-A-p7}
\int_0^1 \varphi_x \big( a|u_{nx}|^2 - \lambda_n^2 |u_n|^2 \big)\, dx   = o(1).
\end{equation}
Multiplying \eqref{pro-A-p7} by $-1/2$ and adding the expression obtained with \eqref{pro-A-p6}, we obtain 
\be
\label{pro-A-p8}
\lambda_n^2 \int_0^1 \varphi_x |u_n|^2 - \int_0^1 \varphi a_x  |u_{nx}|^2\, dx =  o(1) \mbox{ in } \mR.
\ee
Combining \eqref{pro-A-p7} and \eqref{pro-A-p8} yields 
\begin{equation}\label{eq:124}
    \int_0^1 (a \varphi_x - \varphi a_x) |u_{nx}|^2 \, dx  = o(1). 
\end{equation}
Then \eqref{eq:124} holds true for every $C^2$ real-valued function $\varphi$ defined on 
$[0,1]$. Consider now $\varphi(x)=e^{Mx}$ for $x\in [0,1]$ where 
$M=\frac{1+\Vert a_x\Vert_{\infty}}{\underline{a}}$. It is immediate to check that, for a.e. $x\in [0,1]$,
\[
a(x) \varphi_x(x) - \varphi(x) a_x(x)=\varphi(x)(Ma(x)-a_x(x))\geq e^{Mx}(M\underline{a}-\Vert a_x\Vert_{\infty})\geq 1.
\]
From \eqref{eq:124} with such a function $\varphi$, we obtain that
\begin{equation}\label{ff148}
    \int_0^1  |u_{nx}|^2 \, dx =o(1).
\end{equation}

From \eqref{Edotpetito}, \eqref{eq25a}, and \eqref{eq25d}, we obtain a contradiction with the fact 
\begin{equation*}
    \|U_n\|_{\mathcal{H}_0}=o(1),
\end{equation*}
The proof is complete. 
\end{proof}

\section{Stability analysis of a related system}  \label{sect-Related}

In this section we consider dynamical systems associated to \eqref{sys_main_int} in section $2$ as considered in \cite{chitour2023lyapunov}. We establish similar exponential stability results, under the weaker condition that the damping function $q$ is now localized , i.e., it satisfies Assumption~(A2). We will only deal with the system given by a wave equation with second order dynamical boundary conditions at both ends and we prove exponential convergence of solutions towards the attractor. Namely, we consider the following sytem
\begin{subnumcases}{}\label{system_without_integrator}
    u_{tt}-(au_{x})_x=-qu_t,\quad  (x,t)\in[0,1]\times \mathbb{R}_+,\\
    u_t(t,1)=\eta(t), \label{withIntetq}\\
    u_t(t,0)= \zeta(t),\label{withIntxi}\\
    \dot{\eta}(t)=-u_x(t,1)-q_1 \eta(t),\label{NIboundaryeta}\\
    \dot{\zeta}(t)=u_x(t,0)-q_0 \zeta(t), \label{NIboundaryxi}\\
    u(0,\cdot)=u_0,\quad u_t(0,\cdot)=u_1, \\
    \zeta(0)=\zeta_0, \quad \eta(0)=\eta_0. 
\end{subnumcases}

The energy of this system is given by 
\begin{align}\label{energy_without_integrator}
    E(t)=\underbrace{\frac{1}{2}\int_0^1 (u_t^2+au_x^2) \, dx }_{E_0(t)} + \underbrace{\frac{1}{2}\eta(t)^2+\frac{1}{2} \zeta(t)^2}_{E_b(t)}.
\end{align}
Similar to the previous section, along the strong solutions the energy functional $E(\cdot)$ is a non-increasing function of $t$ and its derivative satisfies 
\begin{align}\label{NIEdot}
    \dot{E}(t)=-\int_0^1qu_t^2\, dx-q_1\eta(t)^2-q_0\zeta(t)^2.
\end{align}
System \eqref{system_without_integrator} has been studied in the multidimensional setup  in \cite{buffe2017stabilization}, where the author established logarithmic decay rate using less restrictive hypotheses.
\begin{proposition}
Assume  $(A1)$ and $(A2)$ hold true and the constants $q_0$ and $q_1$ are positive. Then 
\begin{enumerate}
    \item there exist $M, \nu>0$, such that for every weak solution of \eqref{sys_main_int} and every $t \geq 0$, it holds that
\begin{equation*}
    E(t)\leq ME(0)e^{- \nu t},
\end{equation*}
where $E$ is the energy of system \eqref{system_without_integrator} defined in \eqref{energy_without_integrator}.
    
    \item it holds 
    $$
    \max_{x \in [0,1]} |u(t,x)-u_*|^2 \leq ME(0)e^{- \nu t},
    $$
where $u_*$ is given by
\begin{equation}\label{NIFormulaofu_*}
    u_*=\frac{1}{a(0)q_0+a(1)q_1} \left( \int_0^1 \big(u_1(x)+qu_0(x)\big) \, dx +a(0)\Big(\zeta_0+q_0 u_0(0)\Big) + a(1)\Big(\eta_0 +q_1 u_0(1)\Big)     \right).
\end{equation}
\end{enumerate}
\end{proposition}

The proof follows the same lines of the proof of Theorem \ref{th:mainth}. Hence, we give the main steps only.
\begin{proof}
The proof of the first statement is similar to the previous case and we only give the main steps.

Using the multiplier $x\psi a u_x$  and proceding in the same way as in Lemma \ref{lem:S0} we show that 
\begin{equation} \label{PreNIfirstlemma}
\int_S^T E(t) \, dt \leq CE(S)+C \underbrace{\int_S^T \int_{Q_1 \cap(0,1)}(u_t^2+au_x^2) \, dx  \, dt}_{\mathbf{S}_0}.
\end{equation}
Then, using the multiplier $\phi \big( u - u(t,1) \big)$ and followings the computations of Lemma \ref{lemmaS0} we show that 
    \begin{equation}\label{PrNIEQ:eq36}
\mathbf{S}_0 \leq  C \underbrace{\int_S^T \int_{(0,1)\cap Q_2}|u-u(t,1)|^2 \, dx  \, dt}_{\mathbf{S}_1}+C E(S).
\end{equation}

Finally, we consider $v$, the solution of
\begin{equation}\label{PrNIEQ:eequ_v}
\left\{\begin{array}{ll}
(av_{ x})_x=\varphi (u-u(t,1)), \quad & x \in(0,1) \\
v_x(0)=v(0)=0.
\end{array}\right.
\end{equation}
We mention that $v$ satisfies the estimates \eqref{v1andv0}-\eqref{aelemma_vtL}.
Using $v$ as a multiplier and following the same steps as in Lemma \ref{lemma16} we show that for every $\delta_2>0$ we have that
\begin{equation}\label{PrNILEM:eS1}
\mathbf{S}_1 \leq C \delta_2 \int_S^T E(t) \, dt+\frac{C}{\delta_2} E(S).
\end{equation}

Up to choosing $\delta_2$ small enough, we deduce that 
\begin{equation*}
    \int_S^T E(t) \, dt \leq CE(S),
\end{equation*}
and one then concludes by using \cite[Theorem 8.1]{komornik1994exact}. 
\end{proof}

We now prove the second statement. From \eqref{withIntetq} we have that 

\begin{equation}\label{IntegralOfEta}
    u(t,1)-u_0(1)=\int_0^t \eta(s) \, ds,
\end{equation}
since $\eta(\cdot)$ decays exponentially to zero the integral in the right hand converges and so does $u(t,1)$ towards a value that we denote $u_*$. From this it follows that for every $x\in [0,1]$, $u(t,x)$ converges exponentially towards $u_*$. Indeed, as we have that 
\begin{equation}
    \begin{aligned}
\left|u(t, x)-u_*\right|^2 & \leq 2|u(t, x)-u(t, 1)|^2+2 \eta_2^2(t)  \\
& \leq 2 \int_0^1 u_x^2(t, x)\, d x+2 \eta_2^2(t) \\
& \leq \frac{4}{\underline{a}} E(t)+2 \eta_2^2(t)\\
& \leq C E(t),
\end{aligned}
\end{equation}
where $E(\cdot)$ is exponentially convergente towards 0.

We now identify the value of $u_*$.  From \eqref{withIntxi} we have that 
\begin{equation}\label{IntegralOfXi}
    u(t,0)-u_0(0)=\int_0^t \zeta(s)\, ds, 
\end{equation}
since $\zeta(\cdot)$ also decays exponentially to zero the right hand side integral converges as $t$ tends to infinity and we have from  \eqref{IntegralOfEta} and \eqref{IntegralOfXi} that 
\begin{equation}\label{eq:135}
 u_*=\int_0^{+\infty} \eta(t)\, dt+u_0(1)=\int_0^{+\infty} \zeta(t)\, dt +u_0(0).
\end{equation}

Integrating \eqref{NIboundaryeta} and \eqref{NIboundaryxi} with respect to $t$ from $0$ to $+\infty$ and using the results in \eqref{eq:135} we obtain that 

\begin{equation}\label{twoequofu_*}
 u_*=\frac{-1}{q_1}\int_0^{+\infty} u_x(t,1)\, dt +\frac{\eta_0}{q_1} + u_0(1)=\frac{1}{q_0}\int_0^{+\infty} u_x(t,0)\, dt+\frac{\zeta_0}{q_0} +u_0(0).
\end{equation}

On the other hand, we have that 

\begin{equation}\label{formulofJ}
    \begin{aligned}
    J:&= a(1)\int_0^{+\infty} u_x(t,1)\, dt - a(0)\int_0^{+\infty} u_x(t,0)\, dt\\
    &= \int_0^{+\infty} \int_0^1 (au_x)_x \, dx \, dt \\
    &= \int_0^{+\infty} \int_0^1  u_{tt} \, dx \, dt +\int_0^{+\infty} \int_0^1   qu_t \, dx \, dt\\
     &=  \int_0^1\int_0^{+\infty}  u_{tt} \, dt \, dx +\int_0^1\int_0^{+\infty}    qu_t \, dt \, dx\\
     &=- \int_0^1  u_1(x) \, dx -\int_0^1 qu_0(x) \, dx,
    \end{aligned}
\end{equation}
where the last equality holds because

\begin{equation}
    \begin{aligned}
        \int_0^1\int_0^{+\infty}  u_{tt} \, dt \, dx&= \lim_{T\rightarrow + \infty} \int_0^1\int_0^{T}  u_{tt} \, dt \, dx\\
        &= \lim_{T\rightarrow + \infty} \int_0^1 u_t(T,x) \, dx - \int_0^1 u_1(x) \, dx, 
    \end{aligned}
\end{equation}

and 
\begin{equation}
    \begin{aligned}
        \Big|\int_0^1 u_t(T,x) \, dx\Big| \leq \left( \int_0^1 |u_t(T,x)|^2 \, dx \right)^{\frac{1}{2}} \leq C E(T)^\frac{1}{2} \xrightarrow[T\to\infty]{} 0.
    \end{aligned}
\end{equation}

Then, from \eqref{twoequofu_*} and \eqref{formulofJ} we deduce \eqref{NIFormulaofu_*}.

\begin{remark} \rm
    The results of Proposition 2 and Proposition 3 in \cite{chitour2023lyapunov}, which address related systems, remain valid. Their proofs are similar to those presented here and are therefore omitted.
\end{remark}

\section{Analysis  in the  $L^p$ setting - Proof of \Cref{thm-Lp}}

This section consisting of three subsections is devoted to the proof of \Cref{thm-Lp}. In the first and second subsections, 
we establish the well-posedness in the $L^p$-setting for $p \ge 1$ via the D'Alembert formula and via the well-posedness of the corresponding hyperbolic system coupled with a differential system using Riemann invariants. The proof of \Cref{thm-Lp} is given in the third subsection.

\subsection{Well-posedness of \eqref{sys_main_int} in an $L^p$ framework for  $p\in [1,\infty]$ via d'Alembert formula} \label{sect-WP}

We have first of all to define the concept of weak (strong, respectively) solution of \eqref{sys_main_int} in an $L^p$ framework. The main ingredient to achieve such a goal relies on the d'Alembert formula coupled with an appropriate change of functions  yielding Dirichlet Boundary conditions.

The first step consists in considering 
the control system associated with \eqref{sys_main_int} defined as the following inhomogeneous system
\begin{subnumcases}{\label{sys_main_int_cont}}
z_{tt} - z_{xx}  = -qz_t,\quad (t,x)\in \mathbb{R}^+ \times (0, 1), 
\label{sys_main_int:cont} \\
z_t(t,1)=h_1(t),\\
z_t(t,0)=h_0(t),\\
z(0, \cdot) = z_0, \quad z_t (0, \cdot) = z_1 \quad \mbox{ on } (0, 1), 
\end{subnumcases}
where the inputs $h_0,h_1$ belong to $W^{1,p}_{loc}(\mathbb{R}_+,\mathbb{C})$ and the initial conditions $z_0,z_1$ belong to $W^{1,p}([0,1],\mathbb{C})$ and $L^p([0,1],\mathbb{C})$ respectively, with $p\in [1,\infty]$. By performing the change of functions
\[
z(t,x)=w(t,x)+xL_0(t)+(1-x)L_1(t), \quad {(t,x)\in \mathbb{R_+}\times [0,1]},
\]
where $L_0,L_1$ are defined by 
\[
L_0(t)=z_0(1)+\int_0^th_0(\tau)\, d\tau, \quad L_1(t)=z_0(0)+\int_0^th_1(\tau)\, d\tau,
\] 
it is immediate to see that $z$ is solution of \eqref{sys_main_int} if and only if 
$w$ is solution of the inhomogeneous problem on $\mathbb{R}^+ \times (0, 1)$
\begin{subnumcases}{\label{sys_main_int-w}}
w_{tt}(t,x) - w_{xx}(t,x)  = -q(x)w_t(t,x)-x[\dot h_1(t)+q(x)h_1(t)]-(1-x)[\dot h_0(t)+q(x)h_0(t)],
\label{sys_main_int:cont} \\
w(t,1)=0, \ t\geq 0,\\
w(t,0)=0,\ t\geq 0,\\
w(0,x) = z_0(x)-xz_0(1)-(1-x)z_0(0),\ x\in [0, 1],\\
w_t (0,x) = z_1(x)-xh_1(0)-(1-x)h_0(0),\ x\in [0, 1].
\end{subnumcases}
We first define the real-valued functions $e_0,e_1$ on $\mathbb{R}$ as the $2$-periodic of the odd extensions to $[-1,1]$ of the functions $x\mapsto x$ and $x\mapsto 1-x$ respectively, as well as $\widetilde{q}$, the $2$-periodic of the even extension to $[-1,1]$ of the function $q$. 
Then, using the standard reflexion-extension procedure as described in \cite[page 63]{strauss2007} ($2$-periodic of the odd extensions to $[-1,1]$ of $w$, $z_0$, $z_1$ still denoted as they stand), we can now express the solutions of \eqref{sys_main_int-w} almost eveywhere on 
$\mathbb{R}_+\times \mathbb{R}$ 
as 
\begin{align*}
w(t,x)=\frac12[w_0(x+t)+w_0(x-t)]+\frac12\int_{x-t}^{x+t}w_1(s)\, ds\\
+\frac12\int_0^t\int_{x-(t-\tau)}^{x+(t+\tau)}[-\widetilde{q}(s)w_t(\tau,s)+f(\tau,s)]\, d\tau\, ds,
\end{align*}
where the initial conditions are given by
\[
w_0(x)=z_0(x)-e_0(x)z_0(1)-e_1(x)z_0(0),\quad w_1(x)=z_1(x)-e_0(x)h_1(0)-e_1(x)h_0(0),
\]
and the inhomogeneous term $f$ is defined as
\[
f(t,x)=-e_0(x)[\dot h_1(t)+\widetilde{q}(x)h_1(t)]-e_1(x)[\dot h_0(t)+\widetilde{q}(x)h_0(t)], \quad {(t,x)\in \mathbb{R_+}\times [0,1]}.
\]
By an easy computation, one gets that, a.e. on $\mathbb{R}_+\times \mathbb{R}$, 
\begin{align}
z(t,x)&=\frac12[z_0(x+t)+z_0(x-t)]+\frac12\int_{x-t}^{x+t}z_1(s)\, ds
-\frac12\int_0^t\int_{x-(t-\tau)}^{x+(t-\tau)}\widetilde{q}(s)z_t(\tau,s)\, d\tau\, ds\nonumber\\
& +E^0(x,t)z_0(0)+E^1(x,t)z_1(0)\nonumber\\
& +F^0(x,t)h_0(0)+F^1(x,t)h_1(0)+\int_0^t[G^0(x,t,\tau)h_0(\tau)+G^1(x,t,\tau)h_1(\tau)]\, d\tau,\label{eq:ztx}
\end{align}
for some piecewise differentiable $2$-periodic functions $E^0,E^1,F^0,F^1,G^0,G^1$. 

We derive at once the following  expressions for the partial derivative of $z$ namely,
for a.e. on $\mathbb{R}_+\times \mathbb{R}$
\begin{align}
z_t(t,x)&=\frac12[z_0'(x+t)+z_0'(x-t)]+\frac12[z_1(x+t)-z_1(x-t)]\nonumber\\
&-\frac12\int_0^t[\widetilde{q}(x+t-\tau)z_t(t,x+t-\tau)-\widetilde{q}(x-t+\tau)z_t(t,x-t+\tau)]\, d\tau\nonumber\\
& +E^0_t(x,t)z_0(0)+E^1_t(x,t)z_1(0)+F^0_t(x,t)h_0(0)+F^1_t(x,t)h_1(0)\nonumber\\
&+G^0(x,t,t)h_0(t)+G^1(x,t,t)h_1(t)\nonumber\\
&+\int_0^t[G^0_t(x,t,\tau)h_0(\tau)+G^1_t(x,t,\tau)h_1(\tau)]\, d\tau,\label{eq:ztx-t}
\end{align}
and 
\begin{align}
z_x(t,x)&=\frac12[z_0(x+t)+z_0(x-t)]+\frac12[z_1(x+t)+z_1(x-t)]\nonumber\\
&-\frac12\int_0^t[\widetilde{q}(x+t-\tau)z_t(t,x+t-\tau)-\widetilde{q}(x-t+\tau)z_t(t,x-t+\tau)]\, d\tau\nonumber\\
& +E^0_x(x,t)z_0(0)+E^1_x(x,t)z_1(0)+F^0_x(x,t)h_0(0)+F^1_x(x,t)h_1(0)\nonumber\\
&+\int_0^t[G^0_x(x,t,\tau)h_0(\tau)+G^1_x(x,t,\tau)h_1(\tau)]\, d\tau.
\label{eq:ztx-x}
\end{align}
By a fixed point argument as in \cite{ChitourMarxPrieur2020} applied to \eqref{eq:ztx-t} (where $z_t$ is the function whose existence has to be shown), we get existence and uniqueness of weak and strong solutions for \eqref{sys_main_int_cont} as provided next.
\begin{proposition}\label{prop:inhomo-1}
Let $p\in [1,\infty]$. For every $(z_0,z_1)$ in $W^{1,p}([0,1],\mathbb{C})\times L^p([0,1],\mathbb{C})$ ($W^{2,p}([0,1],\mathbb{C})\times W^{1,p}([0,1],\mathbb{C})$ respectively) and $(h_0,h_1)$ in $W^{1,p}_{loc}(\mathbb{R}_+,\mathbb{C})^2$
($W^{2,p}_{loc}(\mathbb{R}_+,\mathbb{C})^2$ respectively),
there exists a unique weak (strong respectively) solution $(z,z_t)$ of \eqref{sys_main_int_cont} which belongs to  $C([0, + \infty); W^{1, p}(0, 1)) \cap C^1([0, + \infty); L^p(0, 1))$ 
($C([0, + \infty); W^{2, p}(0, 1)) \cap C^1([0, + \infty); W^{1, p}(0, 1))$ respectively). 
\end{proposition}
We need the following proposition for subsequent analysis, which is an immediate 
consequence of \eqref{eq:ztx-x}
\begin{lemma}\label{lem:trace}
Let $p\in [1,\infty]$. 
For every $(z_0,z_1)$ in $W^{1,p}([0,1],\mathbb{C})\times L^p([0,1],\mathbb{C})$, $(c_0,c_1)\in\mathbb{C}^2$, consider the application ${\mathcal{T}}$ which associates to every $(h_0,h_1)$ in $W^{1,p}_{loc}(\mathbb{R}_+,\mathbb{C})^2$ so that $h_0(0)=c_0$ and $h_1(0)=c_1$ the pair of functions $(z_x(\cdot,1), z_x(\cdot,0))$. Then ${\mathcal{T}}$ is a linear bounded operator taking values in $L^p_{loc}(\mathbb{R}_+,\mathbb{C})$ and there exists $C_p>0$ (independent of $z_0,z_1,c_0,c_1$) such that, for every $t\in (0,\frac12)$, one has 
\begin{equation}\label{eq:lip}
\Vert {\mathcal{T}}(h_0,h_1)-{\mathcal{T}}(\widetilde{h}_0,\widetilde{h}_1)\Vert_{L^p(0,t)}\leq C_pt(\Vert h_0-\widetilde{h}_0\Vert_{L^p(0,t)}+\Vert h_1-\widetilde{h}_1\Vert_{L^p(0,t)}).
\end{equation}
\end{lemma}

For the second step towards the concept of solution of \eqref{sys_main_int} in an $L^p$ framework, one considers the three-dimensional linear control system
\begin{equation}\label{eq:finitedim}
\dot y=Ay+Bv,\quad y\in\mathbb{C}^3,\quad v\in \mathbb{C}^2,
\end{equation}
with \[
A=\begin{pmatrix}0&1&0\\-\alpha_1&-\alpha_2&0\\0&0&-\gamma_1\end{pmatrix},
\quad  B=\begin{pmatrix}0&0\\-\beta_1&0\\0&\mu_1\end{pmatrix},
\]
 and the (linear) output  map $F$ as 
 \begin{equation}\label{eq:F}
 F(v)=(y_2,y_3).
 \end{equation}
 
Hence $\eta_1,\eta_2,\zeta_1$ defined in \eqref{eq:xi1}, \eqref{eq:eta} and \eqref{eq:xi0} can be seen as solutions of \eqref{eq:finitedim} with control input $v=(u_x(t,1),u_x(t,0))$ and can be written for $t\geq 0$
\begin{equation}\label{eq:int-finite}
\begin{pmatrix}\eta_2(t)\\ \eta_1(t)\\ \zeta_1(t)\end{pmatrix}=e^{At}\begin{pmatrix}\eta_2(0)\\ \eta_1(0)\\ \zeta_1(0)\end{pmatrix}+\int_0^t e^{A(t-\tau)}B
\begin{pmatrix}u_x(\tau,1)\\ u_x(\tau,0)\end{pmatrix}\, d\tau.
\end{equation}
It is now immediate to see that \eqref{sys_main_int} is the interconnection of the two control systems \eqref{sys_main_int_cont} and \eqref{eq:finitedim} where the control inputs
 $(h_0,h_1)$  and $v$ are equal to $(\eta_1,\zeta_1)$ and $(u_x(\cdot,1), u_x(\cdot,0))$ respectively.

 We can now define the notion of solution of  \eqref{sys_main_int} in an $L^p$ framework: a $u$ is a solution of  \eqref{sys_main_int_cont} with initial condition $(z_0,z_1)$ and control functions $(h_0,h_1)$ if and only if $u$ is a solution of  \eqref{sys_main_int} and 
 $(h_0,h_1)=F((u_x(\cdot,1), u_x(\cdot,0)))$. By choosing $p\in [1,\infty]$ and the appropriate functional space, we have a well-defined concept of weak and strong solutions.
 
 \medskip 
 We can now state the existence result for the solution of  \eqref{sys_main_int} in an $L^p$ framework.
 \begin{proposition}\label{prop:existence-p}
 Let $p\in [1,\infty]$. For every $(z_0,z_1)$ in $W^{1,p}([0,1],\mathbb{C})\times L^p([0,1],\mathbb{C})$ ($W^{2,p}([0,1],\mathbb{C})\times W^{1,p}([0,1],\mathbb{C})$ respectively) and $(c_0,c_1,c_2)$ in $\mathbb{C})^3$,
there exists a unique weak (strong respectively) solution $(z,z_t)$ of \eqref{sys_main_int_cont} which belongs to  $C([0, + \infty); W^{1, p}(0, 1)) \cap C^1([0, + \infty); L^p(0, 1))$ \\
($C([0, + \infty); W^{2, p}(0, 1)) \cap C^1([0, + \infty); W^{1, p}(0, 1))$ respectively). 
 \end{proposition}
 \begin{proof}
 With the notations of the proposition, the argument is based on the fact that the definition of solution of  \eqref{sys_main_int} in an $L^p$ framework is equivalent to the existence of a solution of the functional equation $w={\mathcal{T}}(F(w))$, in $L^p([0,T],\mathbb{C}^2)$ for weak solutions and 
 $W^{1,p}([0,T],\mathbb{C}^2)$ for strong solutions, for every $T>0$. 
Here $w$ stands for $(u_x(\cdot,1), u_x(\cdot,0)))$. Thanks to Lemma~\ref{lem:trace}, the affine mapping ${\mathcal{T}}\circ F$ is $\frac12$-Lipschitz for some $t_*>0$ small enough but independent of the data of the problem.  We thus get such solution on $[0,t_*]$. By concatenating such solutions, one gets a global solution on $[0,T]$ for any $T>0$.
  \end{proof}
 
\subsection{Well-posedness of \eqref{sys_main_int} in an $L^p$ framework for $p\in [1,\infty]$ using the Riemann invariants} \label{sect-WP}

In this section, we give another proof on the well-posedness of the solutions in the $L^p$-setting. The analysis is based on the Riemann invariants of the system. Set, for $(t, x) \in [0, + \infty) \times (0, 1)$,  
$$
\rho (t, x) = u_t(t, x) + u_x(t, x) \quad \mbox{ and } \quad \xi (t, x) =  u_t(t, x) - u_x(t, x). 
$$
System \eqref{sys_main_int} becomes (recall that $a \equiv 1$)
\be\label{sys-R-p1}
\left\{ \begin{array}{l}
\xi_t - \xi_x = - 2q (\rho + \xi)  \quad \mbox{ for } (t,x)\in \mathbb{R}_+ \times (0, 1), \\[6pt]
\rho_t + \rho_x = - 2q (\rho + \xi)\quad \mbox{ for } (t,x) \in \mathbb{R}_+ \times (0, 1), \\[6pt]
(\rho + \xi)(t,1)= 2\eta_1(t) \quad \mbox{ for } t\geq 0,\\[6pt]
(\rho + \xi)(t,0)= 2\zeta_1(t) \quad\mbox{ for }  t\geq 0, 
\end{array} \right.
\ee
coupled with an ordinary differential system
\be\label{sys-R-p2}
\left\{ \begin{array}{l}
\dot\eta_1(t)=-\alpha_1\eta_1(t)-\alpha_2\eta_2(t)- \frac{1}{2} \beta_1 (\rho - \xi)(t,1) \quad \mbox{ for } t\geq 0,\\[6pt]
\dot\eta_2(t)=\eta_1(t) \quad \mbox{ for } t\geq 0,\\[6pt]
\dot\zeta_1(t)=-\gamma_1\zeta_1(t)+ \frac{1}{2}\mu_1 (\rho - \xi) (t,0) \quad \mbox{ for } t\geq 0, 
\end{array} \right.
\ee 
with the initial data 
\be \label{sys-R-initial}
\rho(0, \cdot) = \rho_0, \quad \xi (0, \cdot) =\xi_0 \quad \mbox{ in } (0, 1), \quad  \eta_1(0)=\eta_{1,0},\quad \eta_2(0)=\eta_{2,0}, \quad \zeta_1(0)=\zeta_{1,0},  
\ee
where 
$$
\rho_0 = u_1 + u_{0, x} \quad \mbox{ and } \quad \xi_{0} = u_1 - u_{0, x}. 
$$

We will establish the well-posedness of \eqref{sys-R-p1} and \eqref{sys-R-p2}, which will imply the well-posedness in the $L^p$-setting of \eqref{sys_main_int}. The well-posedness of \eqref{sys-R-p1} when $\eta_1$ and $\zeta_1$ are given in $L^p$ is quite standard. 

\begin{proposition}\label{pro-R1} Let $T > 0$, $q \in L^\infty((0, T) \times (0, 1))$, and let $\rho_0, \xi_0 \in L^p(0, 1)$ and $f, g \in L^p(0, T)$. There exists a unique broad solution $(\rho, \xi) \in \Big(C([0, T]; L^p(0, 1)) \cap C([0, 1]; L^p(0, T)) \Big)^2$ of the system  
\be\label{sys-R-p1-1}
\left\{ \begin{array}{l}
\xi_t - \xi_x = - q (\rho + \xi)  \quad \mbox{ for } (t,x)\in (0, T) \times (0, 1), \\[6pt]
\rho_t + \rho_x = q (\rho + \xi)\quad \mbox{ for } (t,x)\in (0, T) \times (0, 1), \\[6pt]
(\rho + \xi)(t,1)= f(t) \quad \mbox{ for } t \in (0, T),\\[6pt]
(\rho + \xi)(t,0)= g(t) \quad\mbox{ for }  t \in (0, T),\\[6pt]
\rho(0, \cdot) = \rho_0, \quad \xi(0, \cdot) = \xi_0 \quad \mbox{ in } (0, 1).  
\end{array} \right.
\ee
Moreover, 
$$
\| (\rho, \xi) \|_{C([0, T]; L^p(0, 1)) \cap C([0, 1]; L^p(0, T))} \le C \left(\| (\rho_0, \xi_0)\|_{L^p(0, 1)} +  \|(f, g) \|_{L^p(0, T)} \right), 
$$
for some positive constant $C$ depending only on $T$ and $\| q\|_{L^\infty((0, T) \times (0, 1))}$. 
\end{proposition}  

The meaning of broad solutions of hyperbolic systems in one dimensional space can be found in, e.g.,  \cite[Chapiter 3]{bressan2000hyperbolic} in the whole line and in \cite{coron2019optimal} where the boundary conditions are involved. The proof is based on standard fixed point argument involving the definition of broad solution and using the following norm: 
\be
\vvvert (\rho, \xi) \vvvert: = \max\Big\{ \sup_{t \in [0, T]} e^{-Lt} \| (\rho, \xi)(t, \cdot) \|_{L^p(0, 1)}, \sup_{x \in [0, 1]} \| (\rho, \xi)(\cdot, x) e^{- L \cdot}  \|_{L^p(0, T)} \Big\} 
\ee
for sufficiently large $L$, the largness depends only on $q$. 
See also \cite[Lemma 3.2]{coron2019optimal}  and \cite[Appendix A]{coron2021optimal} for a more general context in $L^\infty$ and $L^2$ settings and the analysis. The details are omitted. 

We next deal with the well-posedness of \eqref{sys-R-p1} and \eqref{sys-R-p2}. We first note that \eqref{sys-R-p2} can be written under the form
\begin{equation}\label{eq:int-finite}
\begin{pmatrix}\eta_2(t)\\ \eta_1(t)\\ \zeta_1(t)\end{pmatrix}=e^{At}\begin{pmatrix}\eta_2(0)\\ \eta_1(0)\\ \zeta_1(0)\end{pmatrix}+\int_0^t e^{A(t-\tau)}B
\begin{pmatrix} (\rho - \xi)(\tau,1)\\ (\rho - \xi)(\tau,0)\end{pmatrix}\, d\tau, 
\end{equation}
where 
 \[
A=\begin{pmatrix}0&1&0\\-\alpha_1&-\alpha_2&0\\0&0&-\gamma_1\end{pmatrix},
\quad  B=\begin{pmatrix}0&0\\- \frac{1}{2}\beta_1&0\\0& \frac{1}{2}\mu_1\end{pmatrix},
\]
We have, for $0 < T \le 1$, 
\be \label{eq:int-finite-1}
\| h \|_{L^p(0, T)} \le C T \Big(\|(\rho, \xi)(\cdot, 1)\|_{L^p}(0, T)\|+ \|(\rho, \xi)(\cdot, 0)\|_{L^p}(0, T)\| \Big), 
\ee
with 
$$
h(t) =\int_0^t e^{A(t-\tau)}B
\begin{pmatrix} (\rho - \xi)(\tau,1)\\ (\rho - \xi)(\tau,0)\end{pmatrix}\, d\tau.
$$

Combining \eqref{eq:int-finite} and \eqref{eq:int-finite-1}, applying \Cref{pro-R1}, and using a fixed point argument, we obtain the following result. 

\begin{proposition}\label{pro-R2} Let $1 \le p \le + \infty$,  $T > 0$, $q \in L^\infty((0, T) \times (0, 1))$, and let $\rho_0, \xi_0 \in L^p(0, 1)$, $\eta_{1, 0}, \eta_{2, 0}, \xi_{1, 0} \in \mathbb{R}$. There exists a unique weak solution $(\rho, \xi, \eta_1, \eta_2, \zeta_1)$  in  $\Big(C([0, T]; L^p(0, 1)) \cap C([0, 1]; L^p(0, T)) \Big)^2 \times C([0, T])^3$ of the system \eqref{sys-R-p1}, \eqref{sys-R-p2}, and \eqref{sys-R-initial}.  
Moreover, 
$$
\| (\rho, \xi) \|_{C([0, T]; L^p(0, 1)) \cap C([0, 1]; L^p(0, T))} + \|(\eta_1, \eta_2, \zeta_1) \|_{C([0, T])} \le C \left(\| (\rho_0, \xi_0)\|_{L^p(0, 1)} +  |(\eta_{1,0}, \eta_{2, 0}, \zeta_{1, 0})| \right), 
$$
for some positive constant $C$ depending only on $T$ and $\| q\|_{L^\infty((0, T) \times (0, 1))}$. 
\end{proposition}   

\begin{remark} \rm The weak solutions given in \Cref{pro-R2} is understood as follows. The pair $(\rho, \xi)$ is a broad solution of \eqref{sys-R-p1} and the triple $(\eta_1, \eta_2, \zeta_1)$ is a solution of \eqref{sys-R-p2} with the corresponding initial data given by \eqref{sys-R-initial}. 
\end{remark}

As a direct consequence of \Cref{pro-R2}, by considering Riemann invariants, we obtain the following result on the well-posedness of \eqref{sys_main_int} in $L^p$-setting for $1 \le p \le + \infty$. 

\begin{proposition} \label{pro-WP} Let $1 \le p \le + \infty$,  $T > 0$, $q \in L^\infty((0, T) \times (0, 1))$, and let $\rho_0, \xi_0 \in L^p(0, 1)$, $\eta_{1, 0}, \eta_{2, 0}, \xi_{1, 0} \in \mathbb{R}$ There exists a unique weak solution of 
$ (u, \eta_1, \eta_2, \zeta_1) \in X_{p} \times \big(C^0[0, \infty) \big)^3$ of \eqref{sys_main_int} such that 
$$
\partial_t u, \partial_x u \in C([0, 1]; L^p(0, T)) \mbox{ for all } T > 0. 
$$ 
Moreover, 
$$
\| (\partial_t u, \partial_x u) \|_{C^0([0, T]; L^p(0, 1))} + \|(\eta_1, \eta_2, \zeta_1) \|_{C([0, T])} \le C_T \left( \|u_0\|_{W^{1, p}(0, 1)} + \| u_1\|_{L^p(0,1)} +  |(\eta_{1,0}, \eta_{2, 0}, \zeta_{1, 0})| \right),  
$$
for some positive constant $C_T$ independent of the intial data. 
\end{proposition}

\begin{remark} \rm  In \Cref{pro-WP}, a weak considered  solution  means that $u_{tt} (t, x) - u_{xx}(t, x)  = - q(x) u_t(t, x)$ holds in the distributional sense, and  the boundary and the initial conditions are understood as usual thanks to the regularity imposing condition on  the solutions.  
\end{remark}

\begin{remark} \rm It is worth noting that the solutions in the case $p=2$ are the same as the one obtained via the semigroup theory again by the approximation arguments and by the uniqueness in the semigroup setting. 
\end{remark}

\subsection{Proof of \Cref{thm-Lp}}

This section is devoted to the proof of \Cref{thm-Lp}.  One of the main ingredients of its proof is the following lemma. 
\begin{lemma}\label{lem-R} Assume $(A2)$ holds true. Let $1< p < + \infty$.  Consider the system
\be\label{lem-R-sys}
\left\{\begin{array}{c}
\xi_t - \xi_x = - q (\rho + \xi), \\
\rho_t + \rho_x = - q (\rho + \xi), \\
\rho (t, 0)+ \xi (t, 0) = h_1, \\
\rho (t, 1)+ \xi (t, 1) = h_2.  
\end{array}\right.
\ee
There exist positive constants $C_p,\alpha_p$
such that, for $\rho(0, \cdot), \xi(0, \cdot) \in L^p(0, 1)$ and $h_1, h_2 \in L^p(0, + \infty)$, it holds, for the corresponding solution of \eqref{lem-R-sys}, for every $t\geq 0$,  and  for $0 < \eps < 2$, 
\begin{equation}\label{lem-R-Ep}
\hat E_{p}(t) \le C_p e^{-\alpha_p t} \hat E_p(0) + \frac{C_p e^{\eps t}}{\eps^{p-1}} \int_0^t |(h_1, h_2)|^p \, dt. 
\end{equation}
where 
$$
\hat E_p (t) = \int_0^1 (|\rho(t, x)|^p + |\xi(t, x)|^p) \, dx.
$$
\end{lemma}

\begin{proof} We will only consider smooth solutions $(\rho, \xi)$ \footnote{We thus assume that $q$ is smooth. Nevertheless, concerning $q$, the constants in the estimates will depend only on $\| q\|_{L^\infty}$}. The general case will follow by regularizing arguments.

The conclusion is a consequence of \cite[Theorem 4.1]{chitour2024exponential} when $h_1 \equiv 0$ and $h_2 \equiv 0$.  Since the system is linear, it then suffices to consider the case 
\be \label{thm-Lp-u0u1}
(u_1, u_0) = (0, 0), 
\ee
this will be assumed later on.  Multiplying the equation of $\xi$ with $\xi |\xi|^{p-2}$ and the equation of  $\rho$ with $\rho |\rho|^{p-2}$, and integrating the expressions with respect to $x$, after using the boundary conditions, we obtain, for $t > 0$,
\begin{multline*}
\frac{1}{p} \frac{d}{dt} \int_0^1 (|\rho(t, x)|^p + |\xi (t, x)|^p) \, dx + \frac{1}{2} \int_0^1 q (\rho - \xi) (\rho|\rho|^{p-2} - \xi |\xi|^{p-2}) (t, x) \, dx  \\[6pt]
=  - \frac{1}{p} \Big(|\rho(1, t)|^p - |\rho(0, t)|^p \Big) + \frac{1}{p} \Big(|\xi(1, t)|^p - |\xi(0, t)|^p \Big).
\end{multline*}
Using the boundary condition, we derive that  
\begin{multline*}
\frac{1}{p} \frac{d}{dt} \int_0^1 (|\rho(t, x)|^p + |\xi (t, x)|^p) \, dx + \frac{1}{2} \int_0^1 q (\rho - \xi) (\rho|\rho|^{p-2} - \xi |\xi|^{p-2}) (t, x) \, dx  \\[6pt]
=  -  \frac{1}{p} \Big(|\rho(1, t)|^p - |h_1(t) - \xi(0, t)|^p \Big) + \frac{1}{p} \Big(|h_2(t) - \rho(1, t)|^p - |\xi(0, t)|^p \Big).
\end{multline*}
Applying \Cref{lem-ineq} below, we obtain 
$$
\frac{d}{dt} \int_0^1 (|\rho(t, x)|^p + |\xi (t, x)|^p) \, dx \le    \eps \Big(|\rho(1, t)|^p+  |\xi(0, t)|^p \Big) + \frac{C}{\eps^{p-1}}  \Big(|h_1(t)|^p + |h_2(t)|^p \Big).
$$
This yields 
\be \label{thmH-energy}
\int_0^1 (|\rho(t, x)|^p + |\xi (t, x)|^p) \, dx  \\[6pt] 
\le  \frac{C}{\eps^{p-1}} \int_0^t  \Big(|h_1(s)|^p + |h_2(s)|^p \Big) \, ds +  \eps \int_0^t \Big(|\rho(s, 1)|^p+  \xi(s, 0)|^p \Big) \, ds. 
\ee
Using the characteristic, we have 
\be  \label{thmH-energy1}
 \int_0^t \Big(|\rho(s, 1)|^p+  \xi(s, 0)|^p \Big) \, ds \le C  \int_0^t \int_0^1 \Big(|\rho(s, x)|^p+  \xi(s, x)|^p \Big) \, ds \, dx. 
\ee
Set 
$$
f(t)= \int_0^t \int_0^1  (|\rho(s, x)|^p + |\xi (s, x)|^p) \, dx \, ds \quad \mbox{ and } \quad g(t) = \int_0^t  \Big(|h_1(s)|^p + |h_2(s)|^p \Big) \, ds. 
$$
Combining \eqref{thmH-energy} and  \label{thmH-energy1} yields   
$$
f'(t) \le \eps f(t) + \frac{C}{\eps^{p-1}} g(t). 
$$
Since $f(0) = 0$ by \eqref{thm-Lp-u0u1}, it follows  by Gronwall's inequality that, for $t \ge 0$,  
$$
f(t) \le \frac{C}{\eps^{p-1}} \int_0^t e^{\eps (t - s)} g(s) \, ds. 
$$
Since $g$ is increasing, we derive that, for $t \ge 0$,  
$$
\int_0^t e^{\eps (t - s)} g(s) \, ds \le \frac{1}{\eps} g(t) e^{\eps t}. 
$$
Thus 
$$
\eps f(t) \le \frac{C e^{\eps t}}{\eps^{p-1}} g(t).  
$$
The conclusion follows in the case $(u_1, u_0) = (0, 0)$. 
\end{proof}

In the proof of \Cref{lem-R}, we used the following inequality.

\begin{lemma} \label{lem-ineq} Let $1< p<+ \infty$.  There exists $C > 0$ depending only on $p$ such that for $0<\eps < 2$, 
\be \label{lem-ineq-cl2}
|a + b|^p \le (1 + \eps) |a|^p +  \frac{C}{\eps^{p-1}} |b|^p \quad \mbox{ for all } a, b \in \mR.  
\ee
\end{lemma}

\begin{proof} The conclusion follows easily in the case $a=0$. We next assume that $a \neq 0$. By normalizing, without loss of generality, one might assume that $a =1$. One thus needs to prove 
\be \label{lem-ineq-p1}
|1 + b|^p \le (1 + \eps)  +  \frac{C}{\eps^q} |b|^p \quad \mbox{ for all } b \in \mR.  
\ee
The inequality is clear in the case $b \le 0$ since 
$$
|1 + b|\le \max\{1, |b|\} \mbox{ for } b \le 0. 
$$

It remains to deal with the case $0 \le b \le 1$ since the conclusion is clear for $b \ge 1$.

Set 
$$
f(t) =(1 + t)^p \mbox{ for } t \ge 0. 
$$
We have 
$$
f'(t) = p (1 + t)^{p-1}. 
$$
It follows that, for $0 \le b \le 1$,  
$$
(1+ b)^p - 1 = f(b) - f(0) = \int_0^b p (1 + t)^{p-1} \, dt \le p b 2^{p-1} = \frac{p b 2^{p-1}}{ \eps^{\frac{p-1}{p}}} \eps^{\frac{p-1}{p}} 
$$
Applying Young's inequality, we obtain, for $0 \le b \le 1$,  
$$
(1+ b)^p - 1  \le \frac{1}{p} \frac{p^b b^p 2^{p(p-1)}}{\eps^{p-1}} +   \frac{p-1}{p} \eps. 
$$
The conclusion then follows in the case $0 \le b \le 1$. The proof is complete. 
\end{proof}

We are ready to give the proof of \Cref{thm-Lp}. 

\begin{proof}[Proof of \Cref{thm-Lp}]

Applying the result in the case $p=2$, we have 
$$
E_{b, 2} \le C e^{-\alpha t} E(0). 
$$
This implies
$$
|\eta_1(t)|^2 + |\eta_2(t)|^2 + |\zeta_1(t)|^2 \le C e^{-\alpha t} \left( \int_0^1 |u_0(x)|^2 + |u_1(x)|^2 + |\eta_{1, 0}|^2 + |\eta_{2, 0}|^2 + |\xi_{1, 0}|^2 \right). 
$$
Since $p > 2$, we derive that 
\be \label{thm-Lp-p1}
|\eta_1(t)|^p + |\eta_2(t)|^p + |\zeta_1(t)|^p \le C e^{- \beta t} E_p(0)  
\ee
for some $\beta > 0$ and $C > 0$ independent of $t$ and the initial data. 

By \Cref{lem-R}, there exist $C > 0$ and $\alpha > 0$ such that,  for $\eps > 0$, and $t > 0$, 
\begin{multline}
\int_0^1 \Big( |u_t(t, x)|^p + |u_{x}(t, x)|^p \Big) \, dx  \\[6pt]
\le C e^{-\gamma t}  \int_0^1 \Big( |u_t(0, x)|^p + |u_{x}(0, x)|^p \Big) \, dx + C \frac{e^{\eps t}}{\eps^{p-1}} \int_0^t |(\eta_1, \zeta_1)(s)|^p \, ds. 
\end{multline}
This implies  
\begin{multline*}
\int_0^1 \Big( |u_t(2t, x)|^p + |u_{x}(2t, x)|^p \Big) \, dx  \\[6pt] \le C e^{-\gamma t}  \int_0^1 \Big( |u_t(t, x)|^p + |u_{x}(t, x)|^p \Big) \, dx + C \frac{e^{\eps t}}{\eps^{p-1}} \int_t^{2t} |(\eta_1, \zeta_1)(s)|^p \, ds \\[6pt]
\le C e^{-\gamma t} \Big( C e^{-\gamma t}  \int_0^1 \Big( |u_t(0, x)|^p + |u_{x}(0, x)|^p \Big) \, dx + C \frac{e^{\eps t}}{\eps^{p-1}} \int_0^t |(\eta_1, \zeta_1)(s)|^p \, ds \Big) \\[6pt] + C \frac{e^{\eps t}}{\eps^{p-1}} \int_t^{2t} |(\eta_1, \zeta_1)(s)|^p \, ds. 
\end{multline*}
We thus obtain, for $t \ge 0$,  
\begin{multline} \label{thm-Lp-p2}
\int_0^1 \Big( |u_t(2t, x)|^p + |u_{x}(2t, x)|^p \Big) \, dx  
\le C e^{- 2 \gamma t} \int_0^1 \Big( |u_t(0, x)|^p + |u_{x}(0, x)|^p \Big) \, dx \\[6pt] + C e^{-\gamma t} \frac{e^{\eps t}}{\eps^{p-1}} \int_0^t |(\eta_1, \zeta_1)(s)|^p \, ds  + C  \frac{e^{\eps t}}{\eps^{p-1}} \int_t^{2t} |(\eta_1, \zeta_1)(s)|^p \, ds. 
\end{multline}
Using \eqref{thm-Lp-p1}, we derive from \eqref{thm-Lp-p2} that 
\begin{multline}
\int_0^1 \Big( |u_t(2t, x)|^p + |u_{x}(2t, x)|^p \Big) \, dx \le C e^{-2 \gamma t} E_p(0) +  \frac{C}{\eps^{p-1}} e^{-\gamma t + \eps t} E_{p}(0) + C \frac{e^{\eps t }}{\eps^{p-1}} e^{-\beta t} E_p(0) . 
\end{multline}
We thus obtain 
\be \label{thm-Lp-p3}
\int_0^1 \Big( |u_t(2t, x)|^p + |u_{x}(2t, x)|^p \Big) \, dx \le C \Big( e^{-2 \gamma t} + \frac{1}{\eps^{p-1}} e^{-\gamma t + \eps t} + \frac{e^{\eps t }}{\eps^{p-1}} e^{-\beta t}  \Big) E_p(0). 
\ee

Let $\delta_0 > 0$ be a small positive constant defined later. Set 
$$
\eps_0 = \min\{\gamma /2, \beta/ 2, 1\}, 
$$
and fix $T > 0$ such that 
$$
\frac{1}{\eps_0^p} e^{- \eps_0 T} \le \delta_0. 
$$ 
We then derive from \eqref{thm-Lp-p1} and \eqref{thm-Lp-p3} with $\eps = \eps_0$ that 
$$
|\eta_1(2T)|^p + |\eta_2(2T)|^p + |\zeta_1(2T)|^p + \int_0^1 \Big( |u_t(2 T, x)|^p + |u_{x}(2 T, x)|^p \Big) \, dx \le C \delta_0 E_p(0). 
$$
In other words, 
\be
E_p(2T) \le C \delta_0 E_p(0). 
\ee

By choosing $\delta_0$ such that $C \delta_0 < 1/2$, we thus obtain 
\be
E_p(2T) \le \frac{1}{2} E_p(0). 
\ee
The conclusion follows. 
\end{proof}

\addcontentsline{toc}{section}{Acknowledgment} 
\section*{Acknowledgment}
Part of this work was conducted during a visit by Abdelhakim Dahmani to the University of Paris-Saclay. This visit was supported by the EMBRACE Center for International Researchers at Aalen University, a project within the HAW International program of the German Academic Exchange Service (DAAD), funded by the German Federal Ministry of Education and Research (BMBF).

\addcontentsline{toc}{section}{References} 
\bibliography{bib}

\begin{thebibliography}{10}

\bibitem{bardos1992sharp}
C.~Bardos, G.~Lebeau, and J.~Rauch.
\newblock Sharp sufficient conditions for the observation, control, and
  stabilization of waves from the boundary.
\newblock {\em SIAM journal on control and optimization}, 30(5):1024--1065,
  1992.

\bibitem{bohm2014modeling}
M.~B{\"o}hm, M.~Krstic, S.~K{\"u}chler, and O.~Sawodny.
\newblock Modeling and boundary control of a hanging cable immersed in water.
\newblock {\em Journal of Dynamic Systems, Measurement, and Control},
  136(1):011006, 2014.

\bibitem{bressan2000hyperbolic}
A.~Bressan.
\newblock {\em Hyperbolic systems of conservation laws: the one-dimensional
  Cauchy problem}, volume~20.
\newblock OUP Oxford, 2000.

\bibitem{brezis2011functional}
H.~Brezis.
\newblock {\em Functional analysis, Sobolev spaces and partial differential
  equations}.
\newblock Springer, 2011.

\bibitem{buffe2017stabilization}
R.~Buffe.
\newblock Stabilization of the wave equation with ventcel boundary condition.
\newblock {\em Journal de Math{\'e}matiques Pures et Appliqu{\'e}es},
  108(2):207--259, 2017.

\bibitem{cavalcanti2012geo}
M.~M. Cavalcanti, I.~Lasiecka, and D.~Toundykov.
\newblock Geometrically constrained stabilization of wave equations with
  wentzell boundary conditions.
\newblock {\em Applicable Analysis}, 91(8):1427--1452, 2012.

\bibitem{cavalcanti2012wave}
M.~M. Cavalcanti, I.~Lasiecka, and D.~Toundykov.
\newblock Wave equation with damping affecting only a subset of static wentzell
  boundary is uniformly stable.
\newblock {\em Transactions of the American Mathematical Society},
  364(11):5693--5713, 2012.

\bibitem{chen1981note}
G.~Chen.
\newblock A note on the boundary stabilization of the wave equation.
\newblock {\em SIAM Journal on Control and optimization}, 19(1):106--113, 1981.

\bibitem{chen1982mathematical}
G.~Chen and D.~L. Russell.
\newblock A mathematical model for linear elastic systems with structural
  damping.
\newblock {\em Quarterly of Applied Mathematics}, 39(4):433--454, 1982.

\bibitem{ChitourMarxPrieur2020}
Y.~Chitour, S.~Marx, and C.~Prieur.
\newblock {$L^p$}-asymptotic stability analysis of a 1{D} wave equation with a
  nonlinear damping.
\newblock {\em J. Differential Equations}, 269(10):8107--8131, 2020.

\bibitem{chitour2024exponential}
Y.~Chitour and H.-M. Nguyen.
\newblock Exponential decay of solutions of damped wave equations in one
  dimensional space in the lp framework for various boundary conditions.
\newblock {\em ESAIM: Control, Optimisation and Calculus of Variations}, 30:38,
  2024.

\bibitem{chitour2023lyapunov}
Y.~Chitour, H.-M. Nguyen, and C.~Roman.
\newblock Lyapunov functions for linear damped wave equations in
  one-dimensional space with dynamic boundary conditions.
\newblock {\em Automatica}, 167:111754, 2024.

\bibitem{conrad1998strong}
F.~Conrad and A.~Mifdal.
\newblock Strong stability of a model of an overhead crane.
\newblock {\em Control and Cybernetics}, 27(3):363--374, 1998.

\bibitem{coron2019optimal}
J.-M. Coron and H.-M. Nguyen.
\newblock Optimal time for the controllability of linear hyperbolic systems in
  one-dimensional space.
\newblock {\em SIAM Journal on Control and Optimization}, 57(2):1127--1156,
  2019.

\bibitem{coron2021optimal}
J.-M. Coron and H.-M. Nguyen.
\newblock On the optimal controllability time for linear hyperbolic systems
  with time-dependent coefficients.
\newblock {\em Ann. Inst. Fourier, to appear}, 2021.

\bibitem{dafermos_wave_1970}
C.~M. Dafermos.
\newblock Wave {Equations} with {Weak} {Damping}.
\newblock {\em SIAM Journal on Applied Mathematics}, 18(4):759--767, June 1970.

\bibitem{MR2169126}
R.~D\'ager and E.~Zuazua.
\newblock {\em Wave propagation, observation and control in {$1\text{-}d$}
  flexible multi-structures}, volume~50 of {\em Math\'ematiques \& Applications
  (Berlin) [Mathematics \& Applications]}.
\newblock Springer-Verlag, Berlin, 2006.

\bibitem{d2000exponential}
B.~d’Andr{\'e}a Novel and J.-M. Coron.
\newblock Exponential stabilization of an overhead crane with flexible cable
  via a back-stepping approach.
\newblock {\em Automatica}, 36(4):587--593, 2000.

\bibitem{engel2000one}
K.-J. Engel and R.~Nagel.
\newblock {\em One-Parameter Semigroups for Linear Evolution Equations}, volume
  194 of {\em Graduate Texts in Mathematics}.
\newblock Springer-Verlag, New York, 2000.

\bibitem{fourrier2013regularity}
N.~Fourrier and I.~Lasiecka.
\newblock Regularity and stability of a wave equation with a strong damping and
  dynamic boundary conditions.
\newblock {\em Evol. Equ. Control Theory}, 2(4):631--667, 2013.

\bibitem{haraux2018nonlinear}
A.~Haraux et~al.
\newblock {\em Nonlinear vibrations and the wave equation}.
\newblock Springer, 2018.

\bibitem{huang1985}
F.~Huang.
\newblock Characteristic conditions for exponential stability of linear
  dynamical systems in hilbert spaces.
\newblock {\em Ann. of Diff. Eqs.}, 1:43--56, 1985.

\bibitem{huang1988mathematical}
F.~Huang.
\newblock On the mathematical model for linear elastic systems with analytic
  damping.
\newblock {\em SIAM Journal on Control and Optimization}, 26(3):714--724, 1988.

\bibitem{komornik1994exact}
V.~Komornik.
\newblock {\em Exact controllability and stabilization: the multiplier method},
  volume~36.
\newblock Elsevier Masson, 1994.

\bibitem{lee1987stabilization}
E.~B. Lee and Y.~You.
\newblock Stabilization of a hybrid (string/point mass) system.
\newblock In {\em Proc. Fifth Int. Conf. Syst. Eng.(Dayton, Ohio, EUA)}, 1987.

\bibitem{li2017boundary}
C.~Li, J.~Liang, and T.-J. Xiao.
\newblock Boundary stabilization for wave equations with damping only on the
  nonlinear wentzell boundary.
\newblock {\em Nonlinear Analysis}, 164:155--175, 2017.

\bibitem{lions1988controlabilite}
J.-L. Lions.
\newblock Contr{\^o}labilit{\'e} exacte.
\newblock {\em perturbations et stabilisation de systemes distribues}, 1988.

\bibitem{mercier2018indirect}
D.~Mercier, S.~Nicaise, M.~A. Sammoury, and A.~Wehbe.
\newblock Indirect stability of the wave equation with a dynamic boundary
  control.
\newblock {\em Mathematische Nachrichten}, 291(7):1114--1146, 2018.

\bibitem{morgul1994stabilization}
O.~Morgul, B.~P. Rao, and F.~Conrad.
\newblock On the stabilization of a cable with a tip mass.
\newblock {\em IEEE Transactions on automatic control}, 39(10):2140--2145,
  1994.

\bibitem{mugnolo2011damped}
D.~Mugnolo.
\newblock Damped wave equations with dynamic boundary conditions.
\newblock {\em Journal of applied analysis}, 17(2):241--275, 2011.

\bibitem{pruss1984}
J.~Pr{\"u}ss.
\newblock On the spectrum of {$C_0$}-semigroups.
\newblock {\em Transactions of the American Mathematical Society},
  284(2):847--857, 1984.

\bibitem{roman2018boundary}
C.~Roman.
\newblock {\em Boundary control of a wave equation with in-domain damping}.
\newblock PhD thesis, Universit{\'e} Grenoble Alpes, 2018.

\bibitem{roman2022pi}
C.~Roman.
\newblock Pi output feedback for the wave pde with second order dynamical
  boundary conditions.
\newblock In {\em 2022 10th International Conference on Systems and Control
  (ICSC)}, pages 263--270. IEEE, 2022.

\bibitem{russell1969linear}
D.~L. Russell.
\newblock Linear stabilization of the linear oscillator in hilbert space.
\newblock {\em Journal of mathematical Analysis and Applications},
  25(3):663--675, 1969.

\bibitem{strauss2007}
W.~A. Strauss.
\newblock {\em Partial Differential Equations, An Introduction}.
\newblock John Wiley-Sons, 2007.

\bibitem{terrand2019regulation}
A.~Terrand-Jeanne, V.~Andrieu, M.~Tayakout-Fayolle, and V.~D.~S. Martins.
\newblock Regulation of inhomogeneous drilling model with a pi controller.
\newblock {\em IEEE Transactions on Automatic Control}, 65(1):58--71, 2019.

\bibitem{vanspranghe2021velocity}
N.~Vanspranghe, F.~Ferrante, and C.~Prieur.
\newblock Velocity stabilization of a wave equation with a nonlinear dynamic
  boundary condition.
\newblock {\em IEEE Transactions on Automatic Control}, 67(12):6786--6793,
  2021.

\bibitem{zuazua2024exact}
E.~Zuazua.
\newblock {\em Exact Controllability and Stabilization of the Wave Equation}.
\newblock UNITEXT. Springer Cham, 1 edition, 2024.

\end{thebibliography}
\bibliographystyle{abbrv}
\end{document}